\documentclass[12pt]{amsart}

\usepackage{amssymb,amsbsy}
\usepackage{latexsym,array}
\usepackage{verbatim,color}
\usepackage{fullpage,euscript,extarrows}
\usepackage{ytableau}
\usepackage{xcolor}
\usepackage{tikz}
\usetikzlibrary{arrows,shapes,trees}
\usepackage[colorlinks=true,linkcolor=blue,urlcolor=my_color,citecolor=magenta]{hyperref}

\definecolor{my_color}{rgb}{0,0.5,0.5}
\definecolor{MIXT}{rgb}{0.8,0.5,0.2}
\definecolor{mixt}{rgb}{0.5,0.3,0.2}
\definecolor{sin}{rgb}{0,0.5,0.5}
\definecolor{darkblue}{rgb}{0,0.1,0.8}
\definecolor{redi}{rgb}{0.5,0,0.4}

\numberwithin{equation}{section}

\input {cyracc.def}
\font\tencyr=wncyr10 
\font\tencyi=wncyi10 
\font\tencysc=wncysc10 
\def\rus{\tencyr\cyracc}

\def\rusi{\tencyi\cyracc}
\def\rusc{\tencysc\cyracc}

\newtheorem{thm}{Theorem}[section]
\newtheorem{lm}[thm]{Lemma}
\newtheorem{cl}[thm]{Corollary}
\newtheorem{prop}[thm]{Proposition}

\theoremstyle{remark}
\newtheorem{rmk}[thm]{Remark}

\theoremstyle{definition}
\newtheorem{ex}[thm]{Example} 
\newtheorem{df}{Definition}

\newcommand {\be}{{\mathfrak b}}

\newcommand {\g}{{\mathfrak g}}

\newcommand {\h}{{\mathfrak h}}
\newcommand {\ka}{{\mathfrak k}}
\newcommand {\el}{{\mathfrak l}}

\newcommand {\q}{{\mathfrak q}}

\newcommand {\te}{{\mathfrak t}}

\newcommand {\gln}{{\mathfrak{gl}}_n}
\newcommand {\sln}{{\mathfrak{sl}}_n}

\newcommand {\son}{{\mathfrak {so}}_{n}}

\newcommand{\gt}{\mathfrak}

\newcommand {\eus}{\EuScript}

\newcommand {\gA}{{\eus A}}

\newcommand {\gS}{{\eus S}}

\newcommand {\ap}{\alpha}

\newcommand {\lb}{\lambda}
\newcommand {\vp}{\varphi}

\newcommand{\SL}{{\rm SL}}
\newcommand{\GL}{{\rm GL}}
\newcommand{\SO}{{\rm SO}}

\newcommand {\N}{{\mathcal N}}
\newcommand {\co}{{\mathcal O}}

\newcommand {\BC}{{\mathbb C}}

\newcommand {\BR}{{\mathbb R}}

\newcommand {\bL}{{\mathbb L}}

\newcommand {\md}{/\!\!/}

\newcommand {\ad}{{\mathrm{ad\,}}}

\newcommand {\Ann}{{\mathrm{Ann}}}

\newcommand {\gr}{{\mathrm{gr\,}}}

\newcommand {\ind}{{\mathrm{ind\,}}}
\newcommand {\Lie}{{\mathsf{Lie\,}}}

\newcommand {\rk}{{\mathsf{rk\,}}}
\newcommand {\cork}{{\mathsf{cork}}}

\newcommand {\trdeg}{{\mathsf{tr.deg\,}}}

\newcommand {\tri}{\mathfrak{sl}_2}

\newcommand {\ov}{\overline}

\newcommand {\bb}{{\boldsymbol{b}}}
\newcommand {\bl}{{\boldsymbol{\lambda}}}

\newcommand {\beq}{\begin{equation}}
\newcommand {\eeq}{\end{equation}}

\renewcommand{\le}{\leqslant}
\renewcommand{\ge}{\geqslant}

\newcommand {\bbk}{\Bbbk}

\begin{document}
\setlength{\parskip}{3pt plus 2pt minus 0pt}
\hfill { {\color{blue}\scriptsize  February 25, 2019}}
\vskip1ex

\title[Complete integrability on coadjoint orbits]
{
Poisson-commutative subalgebras and complete integrability 
on non-regular coadjoint orbits and  flag varieties}  
\author[D.\,Panyushev]{Dmitri I. Panyushev}
\address[D.\,Panyushev]%
{Institute for Information Transmission Problems of the R.A.S., Bolshoi Karetnyi per. 19, 
Moscow 127051, Russia}
\email{panyushev@iitp.ru}
\author[O.\,Yakimova]{Oksana S.~Yakimova}
\address[O.\,Yakimova]{Universit\"at zu K\"oln,
Mathematisches Institut, Weyertal 86-90, 50931 K\"oln, Deutschland}
\email{yakimova.oksana@uni-koeln.de}
\thanks{The research of the first author was  supported by the Russian Foundation for Sciences.  The second author is
funded by the Deutsche Forschungsgemeinschaft (DFG, German Research Foundation) --- project number 330450448.} 
\keywords{integrable systems, moment map, coisotropic actions,  coadjoint orbits}
\subjclass[2010]{17B63, 14L30, 17B08, 17B20, 22E46}
\dedicatory{To the memory of Bertram Kostant}
\begin{abstract}
The purpose of this paper is to bring together various  loose ends in the theory of integrable systems. 
For a semisimple Lie algebra $\g$, we obtain several results on completeness of homogeneous Poisson-commutative subalgebras of $\gS(\g)$ on coadjoint orbits. This concerns, in particular, Mishchenko--Fomenko and Gelfand--Tsetlin subalgebras. 
\end{abstract}
\maketitle

\section*{Introduction}
\noindent
Symplectic manifolds or varieties $(M,\omega)$ provide a natural setting for integrable systems.  
The algebra of ``suitable'' functions on $M$, $\text{Fun}(M)$,  carries a Poisson bracket, and 
connections with Geometric Representation Theory occur if a  {\it Hamiltonian action\/} of 
a Lie group $Q$ on $M$ is given. 
Let  $\mu\!: M\to \q^*=(\Lie Q)^*$ be the corresponding moment mapping and $\gS(\q)$ the symmetric algebra of $\q$. Then $\gS(\q)$ is a Poisson algebra and 
the co-morphism $\mu^*:\gS(\q)\to \text{Fun}(M)$ is a Poisson homomorphism.
Therefore, if $\eus A\subset\gS(\q)$ is Poisson-commutative, then so is $\mu^*(\eus A)$.
For a {\it coisotropic} Hamiltonian action $(Q, M)$, one obtains a 
completely integrable system on $M$, see~\cite{int}. 
The key point here is the existence of a Poisson-commutative algebra $\eus A\subset\gS(\gt q)$ 
that is {\it complete}, i.e., it provides a complete family in involution on a generic $Q$-orbit in the image of 
$\mu$,  see Definition~\ref{com-fam}.

Two most celebrated examples of Poisson-commutative subalgebras are 
the Gelfand--Tsetlin subalgebras of $\gS(\sln)$ and $\gS(\son)$. 
Their definition goes back to \cite{gt-1,gt-2,gs1,gs2}. 
The success of that construction heavily relies on the existence of chains of coisotropic actions. 
We prove that both these algebras are complete on {\bf every} coadjoint orbit. For arbitrary simple Lie algebras $\g$, a large supply of Poisson-commutative subalgebras of $\gS(\g)$ is given by the argument shift method, see below.

Our ground field $\bbk$ is algebraically closed and of characteristic $0$.   
Let $G$ be a reductive algebraic group over $\bbk$ with $\g=\Lie G$.
Poisson-commutative subalgebras of  $\gS(\g)$ attract a great deal of attention, because of
their relationship to  geometric representation theory. 
If ${\eus A}\subset\gS(\g)$ is Poisson-commutative, then 
$\trdeg \eus A\le \bb(\g)=\frac{1}{2}(\dim\g+\rk\g)$. This is the dimension of a Borel subalgebra of $\g$.
(For arbitrary Lie algebras $\q$, the rank should be replaced with the {\it index}, $\ind\q$.)
In~\cite{mf}, a certain 
Poisson-commutative subalgebra $\eus F_a\subset\eus S(\g)$ is constructed for any $a\in\g^*$. Following~\cite{v:sc}, 
we say that $\eus F_a$ is the {\it Mishchenko--Fomenko subalgebra (associated with $a$)} or just an 
{\it MF-subalgebra}. Say that $a\in\g^*$ is {\it regular\/} if $\dim(Ga)=\dim\g-\rk\g$ and write $\g^*_{\sf reg}$ 
for the set of regular elements. It is known that $\trdeg \eus F_a=\bb(\g)$ if and only if $a\in\g^*_{\sf reg}$. 
The importance of MF-subalgebras and their quantum counterparts 
is advocated e.g. in~\cite{FFR,v:sc,ko09}. 

We prove that, for any $a\in\g^*_{\sf reg}$, $\eus F_a$ is complete on each regular and each closed 
$G$-orbit (Theorem~\ref{thm-s}). The closed orbits are of extreme importance in view of their 
connection with flag varieties and integrable systems related to the compact form of $\g$.   

The crucial r\^ole of nilpotent $G$-orbits is seen in the observation that if an arbitrary homogeneous 
Poisson-commutative subalgebra of $\gS(\g)$ is complete on any nilpotent orbit, then it is complete on 
every orbit, see Proposition~\ref{lm-c-n} and  Corollary~\ref{cl-nilp}. This implies that there is a dense 
open subset $U\subset\g^*_{\sf reg}$ such that $\eus F_a$ ($a\in U$) is complete on {\bf every} 
$G$-orbit, see Proposition~\ref{mf-generic}. Another striking feature is that the question of completeness on regular orbits is reduced to 
the unique regular nilpotent orbit.

The starting point of the Gelfand and Tsetlin construction \cite{gt-1,gt-2} 
for  $\g=\sln$ or $\son$, is 
a chain of Lie algebras
\\[.3ex]
\centerline{
$\g=\g(n)\supset\g(n-1)\supset\dots \supset \g(1)$,} 
\\[.3ex] 
where $\g(k)=\gt{sl}_k$ or $\gt{so}_k$. 
The  {\it Gelfand--Tsetlin (=GT) subalgebra\/} $\hat{\eus C}$  of the enveloping algebra ${\eus U}(\g)$
is generated by the centres of ${\eus U}(\gt g(k))$ with $1\le k\le n$. 
Then $\eus C:=\gr(\hat{\eus C})$ is a 
Poisson-commutative subalgebra of $\gS(\g)$ with $\trdeg\eus C=\bb(\g)$.
The main reason behind many nice features of the GT-subalgebras $\eus C$ is that 
$(\GL_n,\GL_{n-1})$ and $(\SO_n,\SO_{n-1})$ are {\it strong Gelfand pairs}. 
In a certain sense, these are the only strong Gelfand pairs. In Section~\ref{sub-strong}, 
we gather various characterisations of these pairs and explain, in particular,
how coisotropic actions come into play here. 
For $\sln$, it was known for a while that the algebra $\eus C$ is complete on any {\bf regular} $G$-orbit, 
see~\cite[3.8]{KoW}. Recently, this completeness result was obtained in the orthogonal case 
in~\cite{EvC}.  In both cases, we prove that, for {\bf any} $x\in\g$,
$\eus C$ is complete on $Gx$   
and the $G(n{-}1)$-action on $Gx$ is coisotropic.  
Moreover, 
our considerations with nilpotent orbits provide different, simpler proofs in the regular case.


Questions on the completeness of $\eus F_a$ on $Gx\subset\g^*$ are related to the
{\it Elashvili conjecture}, which asserts that $\ind\g^x=\rk\g$ for any $x\in\g^*$.   
In Section~\ref{sect:2}, we report on the current state of this conjecture. 
Theorem~\ref{gt-com} on the completeness of $\eus C\subset \eus U(\sln)$ and the fact that 
this $\eus C$ is a limit of MF-subalgebras~\cite{v:sc}
yield a new proof of Elashvili's conjecture in type {\sf A}, see Remark~\ref{rem-el}({\sf i}). This proof has 
a potential of being generalised to arbitrary $\g$.

Two different geometric features of the Gelfand--Tsetlin construction are 
discovered in~\cite{gs1} and~\cite{KoW}. Guillemin and Sternberg in \cite{gs1} 
work with compact Lie groups over $\BR$ and  exploit a chain of subalgebras 
\[
    \gt u_n \supset \gt u_{n-1} \supset \ldots \supset \gt u_1.
\]
They obtain an integrable system (= complete family of functions), which we call the 
$\bl$-system, see Section~\ref{subs:lambda} for the relation with the GT-subalgebra $\eus C$ in type {\sf A}. 
Briefly speaking, the $\bl$-system is generated by the eigenvalues  
\[
    \{ \bl_k^{[m]}\mid 1\le m<n \ \& \ 1\le k\le n{-}m\}
\]
related to the projections $\gt u_n^*\to \gt u_{n-m}^*$.
This system is examined in details in Section~\ref{subsec-cois}.  
The geometric aspect is that it integrates to an action of a compact torus~\cite{gs1}.  
In~\cite{KoW}, Kostant and Wallach have integrated $\eus C$ to an action of a unipotent group. We hope to explore related geometric properties of MF-subalgebras in a forthcoming article.

In Section~\ref{sect:5}, we study actions of reductive subgroups $H\subset G$ on 
$Gx\subset \gt g^*$. These $H$-actions are obviously Hamiltonian and 
we show that several numerical characteristics of them, 
such as {\it defect\/} and {\it corank},  
are constant along a $G$-sheet $S\subset \g\simeq\g^*$. 
This is very much in the spirit of the useful result that the {\it complexity\/} and {\it rank} of a $G$-orbit 
are constant along any sheet $S\subset\g$, see~\cite[Sect.\,5]{p94}.
Building on the insights of \cite{av}, we prove that the corank does not increase on the closure of a  
sheet, see Theorem~\ref{thm-d}. 
Our completeness result for $\eus C$ in the orthogonal case,  arises as an application of this general theory to the pair $(G,H)= (\SO_{n},\SO_{n-1})$.

\section{Poisson brackets and Mishchenko--Fomenko subalgebras}
\label{sect:prelim}

\noindent
Let $Q$\/ be a connected affine algebraic group with Lie algebra $\q$. The symmetric algebra 
$\gS (\q)$ over $\bbk$ is identified with the graded algebra of polynomial functions on $\q^*$ and we also 
write $\bbk[\q^*]$ for it.  

Let $\q^\xi$ denote the stabiliser in $\q$ of $\xi\in\q^*$. The {\it index of}\/ $\q$, $\ind\q$, is the minimal codimension of $Q$-orbits in $\q^*$. Equivalently,
$\ind\q=\min_{\xi\in\q^*} \dim \q^\xi$. By Rosenlicht's theorem~\cite[2.3]{VP}, one also has
$\ind\q=\trdeg\bbk(\q^*)^Q$. The ``magic number'' associated with $\q$ is $\bb(\q)=(\dim\q+\ind\q)/2$. 
Since the coadjoint orbits are even-dimensional, the magic number is an integer. If $\q$ is reductive, then  
$\ind\q=\rk\q$ and $\bb(\q)$ equals the dimension of a Borel subalgebra. The Poisson--Lie bracket on
$\bbk[\q^*]$ is defined on the elements of degree $1$ (i.e., on $\q$) by $\{x,y\}:=[x,y]$. 
The {\it Poisson centre\/} of $\gS(\q)$ is 
\[
    \gS(\q)^\q=\{H\in \gS(\q)\mid \{H,x\}=0 \ \ \forall x\in\q\} .
\] 
Since $Q$ is connected, we also have $\gS(\q)^\q=\gS(\q)^{Q}=\bbk[\q^*]^Q$.
The set of $Q$-{\it regular\/} elements of $\q^*$ is 
$\q^*_{\sf reg}=\{\eta\in\q^*\mid \dim \q^\eta=\ind\q\}$. Set $\q^*_{\sf sing}=\q^*\setminus \q^*_{\sf reg}$.

Take $\gamma\in\q^*$. Note that $T^*_\gamma \q^* \simeq\q$. Therefore 
the differential $\textsl{d}_\gamma F$  of $F\in\gS(\q)$ can be regarded as an element of $\q$. Let $\hat\gamma=\gamma([\,\,,\,])$ be the skew-symmetric form on $\q$ defined by $\gamma$. 
 In these terms
\begin{equation}\label{P-br}
\{F_1,F_2\}(\gamma)=\gamma([\textsl{d}_\gamma F_1,\textsl{d}_\gamma F_2]) 
=\hat\gamma(\textsl{d}_\gamma F_1,\textsl{d}_\gamma F_2)
\end{equation}
for all $F_1,F_2\in\gS(\q)$. For a subalgebra $\gA\subset\gS(\q)$, set 
$\textsl{d}_\gamma\gA=\left<\textsl{d}_\gamma F\mid F\in\gA\right>_{\bbk}$. Suppose that 
$\gA$ is {\it Poisson-commutative}, i.e., $\{\gA,\gA\}=0$. Then 
$\hat\gamma$ vanishes on $\textsl{d}_\gamma\gA$ for each 
$\gamma\in\q^*$. Clearly $\ker\hat\gamma=\q^\gamma$. Hence  
$\dim \textsl{d}_\gamma\gA \le \dim\q^\gamma+\frac{1}{2}\dim(Q\gamma)$ and 
\beq    \label{tr}
    \trdeg \gA\le \bb(\q).
\eeq
Poisson-commutative subalgebras $\gA$ with $\trdeg\gA=\bb(\q)$ are of particular importance. 

Let $\psi_\gamma\!: T^*_\gamma \q^* \to T^*_\gamma(Q\gamma)$ be the canonical projection. 
Then $\ker\psi_\gamma=\gt q^\gamma$. The skew-symmetric form $\hat\gamma$ is non-degenerate on 
$T^*_\gamma(Q\gamma)$. The algebra $\bbk[Q\gamma]$ carries the Poisson structure, which is defined  
by \eqref{P-br} with $F_1,F_2\in\bbk[Q\gamma]$ and which is inherited from $\q^*$.
Once again, $\{F_1|_{Q\gamma},F_2|_{Q\gamma}\}=\{F_1,F_2\}|_{Q\gamma}$ for all $F_1,F_2\in\gS(\q)$.
The coadjoint orbit $Q\gamma$ is a smooth symplectic variety. 

\begin{df}     \label{com-fam}
A set $\{F_1,\ldots,F_m\}\subset \bbk[Q\gamma]$ is said to be {\it a complete family in involution} if 
$F_1,\ldots,F_m$ are algebraically independent, $\{F_i,F_j\}=0$ for all $i,j$, and 
$m=\frac{1}{2}\dim (Q\gamma)$. 
\end{df}

Let $\gA\subset \gS(\q)$ be a Poisson-commutative subalgebra. Then the restriction of $\gA$ to 
$Q\gamma$, denoted $\gA|_{Q\gamma}$,  
is Poisson-commutative for every $\gamma$.  We say that $\gA$ is {\it complete on\/} $Q\gamma$, if 
$\gA|_{Q\gamma}$ contains a complete family in involution. 
The condition is equivalent to the equality $\trdeg (\gA|_{Q\gamma}) = \frac{1}{2}\dim (Q\gamma)$. 

\begin{lm}      \label{obvious}
Suppose that $\gA\subset \gS(\q)$ is Poisson-commutative, $\gamma\in\q^*_{\sf reg}$, and 
$\dim\textsl{d}_\gamma \gA=\bb(\q)$. Then $\gA$ is complete on $Q\gamma$.
\end{lm}
\begin{proof}
Since $\gamma$ is regular, we have $\dim\ker\psi_\gamma=\ind\q$. 
Therefore 
\[
    \dim \psi_\gamma(\textsl{d}_\gamma \gA) \ge \bb(\q)-\ind\gt q=\frac{1}{2}\dim (Q\gamma)
\]
as required.
\end{proof}

The celebrated ``argument shift method'', which goes back to 
Mishchenko--Fomenko \cite{mf}, provides a large Poisson-commutative subalgebras of $\gS(\q)$ 
starting from the Poisson centre $\gS(\q)^\q$. Given $\gamma\in\q^*$, the $\gamma$-shift of argument 
produces the {\it Mishchenko--Fomenko subalgebra\/} ${\eus F}_\gamma$. Namely, for 
$F\in\gS(\q)=\bbk[\q^*]$, let $\partial_{\gamma} F$ be the direction derivative  of $F$ 
with respect to $\gamma$, i.e., 
\[
    \partial_{\gamma}F(x)=\frac{\textsl{d}}{\textsl{d}t} F(x+t\gamma)\Big|_{t=0}.
\]
Then ${\eus F}_\gamma$ is generated by all $\partial_\gamma^k F$ with $k\ge 0$ and 
$F\in\gS(\q)^{\q}$. The core of this method is that for any $\gamma\in\q^*$ 
there is the Poisson bracket $\{\,\,,\,\}_{\gamma}$ on $\q^*$ such that 
$\{\xi,\eta\}_{\gamma}=\gamma([\xi,\eta])$ for $\xi,\eta\in\q$, and that this new 
bracket is {\it compatible} with $\{\ ,\ \}$. 
Two Poisson brackets on $\gS(\q)$ are said to be {\it compatible}, if all their linear combinations are again 
Poisson brackets.  For more details see \cite[Sect.~1.8.3]{duzu}.

\subsection{Compatible brackets and pencils of skew-symmetric forms}
Take $\gamma\in\q^*$ and let $\eus  F_\gamma$ be the corresponding MF-subalgebra of $\gS(\q)$. 
The original description of $\eus F_\gamma$~\cite{mf} was different  from (but equivalent to) the one 
presented above. 
For $F\in\gS(\q)$ and $t\in\bbk$, let  $F_{\gamma,t}$ be a function on $\q^*$ such that 
$F_{\gamma,t}(x)=F(x+t\gamma)$ for each $x\in\q^*$. Suppose that $\deg F=m$. Then
$F_{\gamma,t}$ expands as a polynomial in $t$ as
\begin{equation}\label{eq-shift}
F_{\gamma,t}
=F^{(0)}+ t F^{(1)}  +\ldots+ t^m F^{(m)} ,
\end{equation}
where   $F^{(k)} = \frac{1}{k!}\partial^k_\gamma F$. As we have stated above, $\eus F_{\gamma}$
is generated by all elements $F^{(k)}$ associated with all  $F\in \gS(\q)^{\q}$.
A standard argument with the Vandermonde determinant  shows that
$\eus F_\gamma$ is generated by $F_{\gamma, t}$ with $F\in\gS(\q)^{\q}$ and 
$t\in\bbk$. It is also clear that if $\gS(\q)^\q$ is generated by $F_1,\dots,F_n$, then $\eus F_\gamma$
is generated by $F_i^{(k)}$ with $i=1,\dots,n$ and all $k$.

Consider the map $\varphi_t\!: \q^* \to \q^*$ such that $\varphi_t (x) = x - t \gamma$ for 
$x\in\q^*$. It extends in the usual way to $\bbk[\q^*]$ and then  
 $F_{\gamma,t}=\varphi_t(F)$. 
The map $\varphi_t$ defines a new Poisson bracket on $\q^*$ by the formula 
\[
      \{F_1,F_2\}_t(x)=\{\varphi_t(F_1),\varphi_t(F_2)\}(\varphi_t(x)),
\]
where  $F_1,F_2\in\bbk[\q^*]$.
For $\xi,\eta\in\q$, the formula reeds 
\[
    \{\xi,\eta\}_t = \vp_t^{-1} (\{\vp_t(\xi),\vp_t(\eta)\}) = [\xi,\eta]-t \hat\gamma(\xi,\eta). 
\]
The Poisson algebras $(\gS(\q),\{\,\,,\,\})$ and $(\gS(\q),\{\,\,,\,\}_t)$ are isomorphic. 
The MF-subalgebra $\eus F_\gamma$ is generated by
$\vp_t^{-1}(\gS(\g)^{\g})$ ($t\in\bbk$), i.e., by  the Poisson centres 
of $(\gS(\q),\{\,\,,\,\}_t)$ with $t\in\bbk$. For $F\in\gS(\g)^{\g}$, we have
\[
        \{F_{\gamma,t}, F_{\gamma,s}\}_t=0=\{F_{\gamma,t}, F_{\gamma,s}\}_s
\] 
and therefore $\{F_{\gamma,t}, F_{\gamma,s}\}=0$ if $t\ne s$. Using the continuity, one concludes that 
$\eus F_\gamma$ is Poisson-commutative.  

Suppose that we wish to calculate $\dim \textsl{d}_x \eus F_\gamma$. The differential
$\textsl{d}_x F_{\gamma,t}=\textsl{d}_{x+t\gamma} F$ lies in the kernel of the skew-symmetric form 
$\hat x_t = \hat x +t \hat\gamma$ if $F\in\gS(\q)^{\q}$.   
Therefore 
\begin{equation}\label{dif-alg}
\textsl{d}_x \eus F_\gamma \subset \sum_{t\in \bbk} \ker (\hat x+t\hat\gamma).
\end{equation} 

We consider below the following conditions on $\q, x,\gamma$:
\beq  \label{eq-cond}
\begin{cases}  (1) & \trdeg \gS(\q)^{\q}=\ind\q,  \\
(2) & (x+\bbk \gamma) \cap \q^*_{\sf reg}\ne\varnothing \\
(3) & \text{there is at least one }  \lb\in\bbk  \text{ such that } \textsl{d}_{y} (\gS(\q)^{\q})=\gt q^{y} 
  \text{ for } y=x+\lb \gamma.
 \end{cases} 
\eeq
Note that  (3) implies (1) and (2). There are tricks that allow one to lift (1), but we are not going to consider them.  Condition (2) is quite harmless, it is satisfied if $\gamma\in\q^*_{\sf reg}$ or $x\in\q^*_{\sf reg}$.  

\begin{lm}       \label{lm-dif-alg}
Suppose that (3) of \eqref{eq-cond} holds. Then 
$$\textsl{d}_x \eus F_\gamma = \sum_{x+ t\gamma\in \q^*_{\sf reg}} \ker (\hat x+ t\hat\gamma)=:L(x,\gamma).$$
\end{lm}
\begin{proof}
Condition (3) implies that there is a non-empty open subset $Y\subset (x+\bbk \gamma)$ such that 
$\textsl{d}_{y} (\gS(\q)^{\q})=\gt q^{y}=\ker\hat y$   for all $y\in Y$.  
Thus, $\sum_{y\in Y} \gt q^y \subset \textsl{d}_x \eus F_\gamma$.  

For almost all $t\in\bbk$, we have $x+t\gamma\in\q^*_{\sf reg}$.  
If $x'=x+t_0\gamma\in\q^*_{\sf sing}$, then nevertheless 
$\textsl{d}_{x'} F_{\gamma,t_0}=\lim_{t\to t_0} \textsl{d}_x F_{\gamma,t}$,
where we can assume that $x+t\gamma\in\q^*_{\sf reg}$. Here  $\textsl{d}_x F_{\gamma,t}\in L(x,\gamma)$ and hence $\textsl{d}_{x'} F_{\gamma,t_0}\in L(x,\gamma)$ as well. 
Now we have 
\[
   \sum_{y\in Y} \gt q^y \subset \textsl{d}_x \eus F_\gamma \subset L(x,\gamma).
\]
According to  \cite[Lemma~A.1]{mrl}, $\sum_{y\in Y} \gt q^y=L(x,\gamma)$.
This concludes the proof.  
\end{proof}

Assume also that $\hat x$ and $\hat \gamma$ are not proportional. 
Now the problem  is to deal with the pencil of skew-symmetric forms on $\q$ generated by 
$\hat x$ and $\hat\gamma$. 

Let $\eus P$ be a two-dimensional vector space of (possibly degenerate) skew-symmetric bilinear
forms on a finite-dimensional vector space $V$. Set $m=\max_{A\in \eus P }\rk A$, and let 
$\eus P_{\sf reg}\subset \eus P$ be the set of all forms of rank $m$. Then $\eus P_{\sf reg}$ is a conical
open subset of $\eus P$.
For each $A\in \eus P$, let $\ker A\subset V$ be the kernel of $A$. 
Our object of interest is the subspace 
$L:=\sum_{A\in \eus P_{\sf reg}} \ker A$.  

\begin{prop}[see {\cite[Thm 1(d)]{JK}}]      \label{prop-JK-c}
Take non-proportional $A,B\in  \eus P_{\sf reg}$. Then 
there is the so-called {\it Jordan--Kronecker canonical form\/} of $A$ and $B$. 
Namely, $V=V_1\oplus\ldots\oplus V_d$, where $A(V_i,V_j)=0=B(V_i,V_j)$ for $i\ne j$, and 
accordingly, $A=\sum A_i$ and $B=\sum B_i$. There are two possibilities for 
$(A_i,B_i)$, one obtains either a {\it Kronecker} or a {\it Jordan block\/} here, see figures below. 
Assume that $\dim V_i>0$ for each $i$. 

\vskip-10pt
\[
\hskip-30pt
\begin{array}{lcc}
  &  A_i  &  B_i     \\  & & \\
\begin{array}{c} \text{A Jordan block } \\ (\lambda_i\in \bbk^\times) \end{array}: &
\begin{pmatrix}
     &   \!\!\! \eus J(\lambda_i) \\  
   \!  -\eus J^\top(\lambda_i) &
 \end{pmatrix} &
 \begin{pmatrix}
   &  -\mathrm{I}  \\  
   \mathrm{I} &
 \end{pmatrix} , \\     & &
 \\

\begin{array}{c} \text{a Kronecker} \\
\text{block} \end{array}: &
 \!\!\!\!\!\!\!\!\!\!\!\! \begin{pmatrix}
     &
\hskip-10pt \boxed{ \begin{matrix}
  1 & 0  & & \\
      &\ddots & \ddots & \\
     &           &  1 & 0
     \end{matrix}}
     \\
\boxed{\begin{matrix}
\!\!\! -1  &             &   \\
  0 & \ddots &   \\
     & \ddots &\!\!\! -1 \\
     &             &0
  \end{matrix}
  }\end{pmatrix}
 &
  \begin{pmatrix}
     &
\hskip-10pt \boxed{ \begin{matrix}
  0 & 1  & & \\
      &\ddots & \ddots & \\
     &           &  0 & 1
     \end{matrix}}
     \\
\boxed{\begin{matrix}
 0  &             &   \\
  \!\!\! -1 & \ddots &   \\
     & \ddots &0 \\
     &             &\!\!\! -1
  \end{matrix}
  }\end{pmatrix},
\end{array}
\]
where 
$\eus J(\lambda_i)=\begin{pmatrix}
\lambda_i & 1 & &\\
& \lambda_i & \ddots & \\
& & \ddots & 1 \\ & & & \lambda_i
\end{pmatrix}.
$   
\qed
\end{prop}
\noindent
{\bf Remark.} In general, there can occur ``Jordan blocks with $\lb_i=\infty$'', but this is not the case here, since $B\in\eus P_{\sf reg}$. Since $A\in\eus P_{\sf reg}$ as well, the case of $\lambda_i=0$ doesn't occur either. 

\begin{prop}    \label{prop-JK}
{\sf (i)} For each non-zero $C\in\eus P$, we have $\dim (L\cap \ker C)=\dim V-m$. \\[.2ex]
{\sf (ii)} If $\eus P_{\sf reg}=\eus P\setminus\{0\}$, then $\dim L=\dim V-\frac{m}{2}$. \\[.2ex]
{\sf (iii)} Suppose that $C\in\eus P$, $C\ne 0$, and $C\not\in \eus P_{\sf reg}$.
Then $\dim L\le (\dim V-m)+\frac{1}{2}\rk C$ and $\dim L = (\dim V-m)+\frac{1}{2}\rk C$
if and only if $\eus P\setminus\eus P_{\sf reg}=\bbk C$, 
$\rk (A|_{\ker C})=\dim\ker C-\dim V+m$ for $A\in\eus P_{\sf reg}$. 
\end{prop}
\begin{proof}
We choose non-proportional $A,B\in\eus P_{\sf reg}$ and bring them into 
a Jordan--Kronecker form according to Proposition~\ref{prop-JK-c}. 
Keep the above notation. In particular, $V=V_1\oplus\ldots \oplus V_d$.  
For any $C\in\eus P$, we have $C=\sum C_i$ accordingly.   

Note that if $V_i$ gives rise to a Jordan block, then $\dim V_i$ is even and  
both $A_i$ and $B_i$ are non-degenerate on $V_i$. For a Kronecker block, 
$\dim V_i=2k_i+1$, $\rk A_i=2k_i=\rk B_i$ and the same holds for every non-zero linear combination 
of $A_i$ and $B_i$. 

Let us assume that $V_i$ defines a Kronecker block if and only if $1\le i\le d'$. Then necessarily 
$d'=\dim V-m$. 
We have 
$$
   L=\bigoplus_{i=1}^{d'}\sum_{C\in\eus P_{\sf reg}} \ker C_i=: \bigoplus_{i=1}^{d'} L_i.
$$ 
It follows from the matrix form of a Kronecker block that $L_i$ is the linear span of the last $(k+1)$ 
vectors in the basis of $V_i$. Hence $\dim L_i=k_i+ 1$. 
For any non-zero $C\in \eus P$, we have $\ker C\cap L=\bigoplus_{i=1}^{d'} (\ker C\cap L_i)$, also   
$\dim\ker C_i=1$ and $\ker C_i\subset L_i$ for each $i\le d'$. Thereby  $\dim(\ker C\cap L)=d'$. 
Thus,  {\sf (i)} is settled. 

If $\lambda=\lambda_i$ for $\lambda_i$ coming from a Jordan block, then $C=A+\lambda B\not\in\eus P_{\sf reg}$ and $C\ne 0$.  Hence the equality $\eus P\setminus\{0\}=\eus P_{\sf reg}$ takes place if and only if there are no Jordan blocks. In this case $\dim L=(\dim V+d)/2$. Part  {\sf (ii)} is settled as well.

{\sf (iii)} By the assumptions on $C$, up to a non-zero scalar factor $C=A+\lambda_i B$,
where  $\lambda_i$ comes from a Jordan block. 
We have $\dim L=d' + \frac{1}{2}\sum_{j=1}^{d'} \rk C_j$. Clearly 
$\sum_{j=1}^{d'} \rk C_j \le \rk C$. The equality takes place if and only if $C_j=0$ for $j>d'$. 
Further, $C_j=0$ if and only if $\lambda_j=\lambda_i$ and $\dim V_j=2$. The first condition, 
$\lambda_i=\lambda_j$, is satisfied if and only if $\eus P\setminus \eus P_{\sf reg}=\bbk C$.  
Until the end of the proof assume that $\lambda_i=\lambda_j$ for all $j>d'$. 

Set $U=\ker C$.
Note that $A$ and $C$ generate $\eus P$. Therefore 
$\rk(A'|_{U})=\rk(A|_{U})$ for every $A'\in\eus P_{\sf reg}$. Recall that $\dim\ker C_j=1$ if $j\le d'$. 
Since $U=\bigoplus_{j=1}^d \ker C_j$ and the spaces $\{\ker C_j\}$ are pairwise orthogonal w.r.t. any
form in $\eus P$, we have $A(\ker C_j,U)=0$ for $j\le d'$. Hence the condition 
$\rk(A|_U)=\dim U-\dim V+ m$ implies that $A_j$ is non-degenerate on 
$\ker C_j$ for any $j>d'$. The explicit matrix form of a Jordan block shows that $\ker C_j$ is 
spanned by two middle basis vectors of $V_j$. Therefore, $A_j$ is non-degenerate on 
$\ker C_j$ if and only if  $\dim V_j=2$. 
This completes the proof. 
\end{proof}

\begin{cl}             \label{cl-MF} 
Suppose that (3) of \eqref{eq-cond} holds for $x$ and $\gamma$. 
Then $\dim(\textsl{d}_x \eus F_\gamma \cap \gt q^x)=\ind\q$ and  
$\dim \textsl{d}_x \eus F_\gamma \le \ind\gt q+\frac{1}{2}\dim(Qx)$. 
Assume additionally that $\hat x$ and $\hat\gamma$ are non-proportional.  
Then 
\[
    \dim \textsl{d}_x \eus F_\gamma = \ind\gt q+\frac{1}{2}\dim(Qx)
\]     
if and only if $(\bbk x  \oplus \bbk \gamma) \cap \q^*_{\sf sing} \subset \bbk x$ and 
$\dim (\gt q^x)^{\bar\gamma}=\ind\q$ for the restriction $\bar\gamma=\gamma|_{\gt q^x}$. 
\end{cl}
\begin{proof}
Consider first the case, where  $\dim(\bbk \hat x+\bbk\hat\gamma)\le 1$. Suppose that (3) holds for 
$y\in x+t\gamma$.
Then $\textsl{d}_x \eus F_\gamma= \textsl{d}_y (\gS(\q)^{\q})=\gt q^y$.  
Here $y$ is necessary regular and $\dim\gt q^y=\ind\q$. 

Suppose now that $\hat x$ and $\hat\gamma$ are non-proportional. 
By Lemma~\ref{lm-dif-alg}, $\textsl{d}_x \eus F_\gamma=L(x,\gamma)$, where    
$L(x,\gamma)=\sum_{\hat y\in\eus P_{\sf reg}} \ker\hat y$ for $\eus P=\bbk\hat x\oplus\bbk\hat\gamma$. 
According to  Proposition~\ref{prop-JK},  we have 
$$
\dim(\ker\hat x\cap L(x,\gamma))=\ind\gt q
\ \text{ and } \ 
\dim L(x,\gamma) \le \ind\gt q+\frac{1}{2}\dim(Qx).
$$
 By the same proposition, the inequality turns into equality  if and only if $\eus P\setminus\eus P_{\sf reg}\subset \bbk\hat x$ and  
$\dim(\gt q^x)^{\bar\gamma}=\ind\q$ in case $x\in\q^*_{\sf sing}$. Note that 
$\gt q^y$ is Abelian for any $y\in\q^*_{\sf reg}$, see e.g. \cite[Sect.~1]{p03}, and therefore 
$\ind\gt q^y=\dim\gt q^y=\ind\q$, 
$(\gt q^y)^*_{\sf reg}=(\gt q^y)^*$ in this case. 
\end{proof}

\begin{rmk}
{\sf (i)} An idea how to estimate $\trdeg (\eus F_\gamma\vert_{Qx})$ appeared in \cite{bols}, see also \cite{bol-TG}, especially for the use of Jordan--Kronecker blocks. \\[.2ex]
{\sf (ii)} The Poisson-commutativity of $\eus F_\gamma$ can be shown using pencils of 
skew-symmetric forms. 
The equality $\{\eus F_\gamma,\eus F_\gamma\}=0$ holds if and only if 
$\hat x(\textsl{d}_x \eus F_\gamma,\textsl{d}_x\eus F_\gamma)=0$ for generic $x\in \q^*$.
In case $\gamma=0$, we have $\eus F_0=\gS(\q)^{\q}$ and there is nothing to prove.   
Suppose that $x\in\q^*_{\sf reg}$ and that $\hat\gamma$ and $\hat x$ are non-proportional.
By the same continuity principle, which has been used in the proof of Lemma~\ref{lm-dif-alg}, 
$\textsl{d}_x \eus F_\gamma\subset L(x,\gamma)$. Suppose that $\xi\in\ker (\hat x+\lambda \hat\gamma)\subset L(x,\gamma)$. Making use of \cite[Lemma~A.1]{mrl}, one writes
\[
   L(x,\gamma)=\sum_{x+ t\gamma\q^*_{\sf reg},\,t\ne \lambda} \ker(\hat x+t\hat\gamma). 
\]
Let $\eta\in\ker(\hat x+\mu\hat\gamma)\subset L(x,\gamma)$ with $\mu\ne \lambda$. Then 
\[
  (\mu-\lambda)\hat x(\xi,\eta)=\mu(\hat x+\lambda\hat\gamma)(\xi,\eta)-\lambda(\hat x+\mu\hat\gamma)(\xi,\eta)=0+0=0.  
\]
Thus, $\hat x(\xi,\eta)=0$ and $\hat x$ vanishes on $\textsl{d}_x\eus F_\gamma$. 
\end{rmk}

\section{Complete subalgebras and nilpotent orbits}
\label{sect:2}

\noindent
In this section, $G$ is a connected reductive $\bbk$-group and $\g=\Lie G$. 
Set $l=\ind\g=\rk\g$. By a classical result of Chevalley, 
$\gS(\g)^{\g}=\bbk[H_1,\ldots,H_l]$, where the $H_i$'s are homogeneous and 
algebraically independent.  Furthermore, $\sum_{j=1}^l \deg H_j = \bb(\g)$. 
Take $a\in\g^*$.  Recall that the MF-subalgebra $\eus F_a\subset \gS(\g)$
is generated by the direction derivatives
$\partial_a^k H_i$ with $1\le i\le l$ and $0\le k\le \deg H_i{-}1$.

Fix an isomorphism $\g^* \simeq \g$ of $G$-modules.  
Making use of this isomorphism, we transfer 
the standard terminology for $\g$ to the elements of $\g^*$, 
e.g. while referring to nilpotent and semisimple elements of  $\g^*$, considering sheets, etc. 

Our main concern in this section  is the following question:
\\[.5ex]
\centerline{\it
Is $\eus F_a$ complete on an orbit $Gx\subset\g^*$? }
\\[.5ex]
For $Gx=\{x\}$, any choice of $a$  leads to a complete 
subalgebra. Therefore we consider only $Gx$ with    $\dim(Gx)\ge 2$. 
It is reasonable to assume that $a\in\g^*_{\sf reg}$. 
Whenever computing $\dim\textsl{d}_x \eus F_a$ we will suppose that 
$\hat a$ and $\hat x$ are non-proportional.  This can be achieved by 
taking some other $x'\in Gx$ instead of $x$. 

\begin{lm}               \label{comp-semi}
Take $a\in\g^*_{\sf reg}$. Then $\dim (\textsl{d}_x \eus F_a\cap\g^x)=l$ for each $x\in\g^*$.
Furthermore, $\eus F_a$ is complete on $Gy\ne \{y\}$ if and 
only if \,$\ind\g^y=l$ and there is $x\in Gy$ such that 
\begin{itemize}
\item[{\sf (i)}]   $(\bbk x  \oplus \bbk a) \cap \g^*_{\sf sing} \subset \bbk x$, 
\item[{\sf (ii)}] ${\bar a}\in(\g^{x})^*_{\sf reg}$ \   for the restriction $\bar a=a|_{\g^{x}}$. 
\end{itemize}  
\end{lm}
\begin{proof}
First, let us examine the  conditions in \eqref{eq-cond}.  
Clearly, $\trdeg \gS(\g)^{\g}=\ind\g$. Since $a$ is regular, (2) holds as well. 
By the {\it Kostant regularity criterion\/} \cite[Thm~9]{ko63}, 
\beq     \label{eq:ko-re-cr}
\text{ $\left<\textsl{d}_\xi H_j \mid 1\le j\le l\right>_{\bbk}=\g^\xi$ \ if and only if \ $\xi\in\g^*_{\sf reg}$.} 
\eeq
Hence (2) implies (3). Now we are ready to use Corollary~\ref{cl-MF}.
It asserts, in particular, that $\dim(\eus F_a\cap \g^x)=l$ for each $x\in\g^*$.
In view of this, $\eus F_a$ is complete on $Gy$ if and only if there is 
$x\in Gy$ such that $\dim\textsl{d}_x\eus F_a=l+\frac{1}{2}\dim (Gx)$. W.l.o.g. assume that 
$\hat x$ and $\hat a$ are non-proportional. Then   by Corollary~\ref{cl-MF}, the 
equality $\dim\textsl{d}_x\eus F_a=l+\frac{1}{2}\dim (Gx)$ takes place if and only if 
$(\bbk x  \oplus \bbk a) \cap \g^*_{\sf sing} \subset \bbk x$ and 
$\dim(\g^x)^{\bar a}=l$.  

Consider the condition $\dim(\g^x)^{\bar a}=l$. It implies that $\ind\g^x\le l$.
At the same time $\ind\g^x\ge \ind\g$ by Vinberg's inequality, see~\cite[Cor.\,1.7]{p03}. 
If this condition is satisfied, then $\ind\g^x=l$. In the other direction, if $\ind\g^x=l$, 
then $\dim(\g^x)^{\bar a}=l$ if and only if $\bar a\in(\g^x)^*_{\sf reg}$.
\end{proof}

\begin{cl}      \label{equiv}
Keep the assumption $a\in\g^*_{\sf reg}$. 
Then $\eus F_a$ is complete on $Gy$ if and 
only if there is $x\in Gy$ such that $\dim\textsl{d}_x \eus F_a = l + \frac{1}{2}\dim (Gx)$.
\qed 
\end{cl} 

The assertion 
\[
     \ind\g^x=\rk\g \ \text{ for each } \ x\in\g
\]
is known as {\it Elashvili's conjecture}. It has no fully conceptual proof in spite of many efforts. 
However, the equality  obviously holds for all regular and all semisimple elements.  Elashvili's conjecture 
is proven for the classical Lie algebras \cite{fan}  and for all Richardson elements~\cite{CM}.
It is also checked for the exceptional $\g$~\cite{graaf,CM}.  
We take it for granted that Elashvili's conjecture is true. 
Therefore, for any orbit $Gx\subset \g^*$, there is an element $a\in\g^*_{\sf reg}$ such that 
the MF-subalgebra $\eus F_a$ is complete on $Gx$, see~\cite{bols} and also~\cite[Sect.~2]{m-y}.  

Return for a while to an arbitrary algebraic Lie algebra $\q=\Lie Q$.  
Take $a,x\in\q^*$ and  let $F\in\gS(\q)$ be a homogeneous polynomial of degree $d$. 
Then 
\begin{equation}\label{ref-p}
  (d-k-1)! \,\textsl{d}_x (\partial^k_a F) = (k-1)! \,\textsl{d}_a (\partial_x^{d-k-1} F)
\end{equation}
and therefore 
\begin{equation}\label{ref-dif}
  \textsl{d}_x \eus F_a = \textsl{d}_a \eus F_x 
\end{equation}
as a subspace of $\q$. 

\begin{thm}        \label{lm-dual}
Suppose that  $a,x\in\gt q^*_{\sf reg}$ and  that $\q$, $\gamma=a$, and $x$ satisfy \eqref{eq-cond}. Then
$\eus F_a$ is complete on $Qx$ if and only if $\eus F_x$ is complete on $Qa$. 
\end{thm}
\begin{proof}
Clearly  \eqref{eq-cond} holds for $a$ and generic points $x'\in Qx$. 
Suppose that $\eus F_a$ is complete on $Qx$. 
By Lemma~\ref{obvious} and Corollary~\ref{cl-MF}, 
this is the case if and only if 
there is $q\in Q$ such that $\dim\textsl{d}_{qx} \eus F_a=\bb(\q)$.  
As one can easily see, $q\textsl{d}_x \eus F_a = \textsl{d}_{qx} \eus F_{qa}$.
Combining this $Q$-equivariance with \eqref{ref-dif}, we conclude that 
$\dim\textsl{d}_{q^{-1}a} \eus F_x = \dim \textsl{d}_{qx} \eus F_{a}=\bb(\q)$.
The equality $\dim\textsl{d}_{q^{-1}a} \eus F_x=\bb(\q)$ implies that $\eus F_x$ is complete on 
$Qa$, see Lemma~\ref{obvious}. 
\end{proof}

\subsection{}
By a result of Tarasov~\cite{t:max}, if $a\in\g^*_{\sf reg}$ is  semisimple, then $\eus F_a$ is complete on 
{\sl every} coadjoint orbit $Gx\subset \g^*_{\sf reg}$.  See also \cite{ko09} for its applications. 
As the next step, we lift the assumption that $a$ is semisimple and also
allow $x$ to be regular or semisimple. 

\begin{thm}       \label{thm-s}
Let $a\in\g^*_{\sf reg}$. The MF-subalgebra $\eus F_a$ is complete on $Gx$ whenever $x$ is semisimple 
or regular. In other words,  $\eus F_a$ is complete on each closed or regular (co)adjoint orbit.
\end{thm}
\begin{proof}
Let $\{e,h,f\}$ be a principal $\tri$-triple in $\g$ and $\be=\Lie B$ be the unique Borel subalgebra that contains $e$.
Then $\g^e\subset\be$ and $\eus K=f+\g^e$ is the associated {\it Kostant section\/} in $\g\simeq \g^*$. 
By \cite{ko63}, $G \eus K=\g^*_{\sf reg}$. Clearly $G\gt b=\g$. 
W.l.o.g. we may assume that $a=f+y\in \eus K$. Take $x\in\be$. 

Suppose that $x$ is semisimple or regular. In the first case, 
$\g^x$ is reductive and clearly
$\ind\g^x=\rk\g^x=\rk\g$. In the second, $\dim\g^x=l=\ind\g^x$. 
Now it suffices to verify conditions {\sf (i)} and {\sf (ii)} of  
Lemma~\ref{comp-semi}  for 
the pair $(a,x)$.  Note that {\sf (ii)} holds for each $a\in\g^*$ if $x$ is regular. 

{\sf (i)} \ A generic element of the plane $\langle a,x\rangle_\bbk$ is of the form 
$\ap(f+y)+ \beta x=\ap f +(\ap y+\beta x)$, where $y,x\in\be$. If $\ap\ne 0$, then all these elements are 
regular in $\g^*$, in view of a classical result of Kostant. Indeed, he proved that 
$f+\be\subset \g_{\sf reg}$, see~\cite{ko63}.

{\sf (ii)} \ Under the assumption that $x$ is semisimple, we have $x\in B\gt t$,
where $\gt t=\g^h\subset\be$ is a Cartan subalgebra.  W.l.o.g. assume that $x\in\te$.  
Then $\g^x=\el$ is a standard Levi subalgebra. 
Further,  
$\bar f=f\vert_\el$   is a regular nilpotent element of $\gt l$ and it can be included into a 
principal   $\gt{sl}_2$-triple $\{\tilde e, \tilde h, \bar f\}\subset \gt l$ such that $\tilde h\in\gt t$. 
Note that 
$\gt l\cap \gt b$ is the unique Borel subalgebra of $\gt l$ containing $\tilde e$. 
We have  $\bar a=\bar f+ \bar y\in\el^*\simeq\el$, where $\bar y\in  \el\cap\be$.
By the same result of Kostant \cite{ko63}, $\bar f+(\el\cap\be)\subset \gt l_{\sf reg}$, and therefore 
$\bar a\in\el^*_{\sf reg}$.
\end{proof}

One is tempted to generalise Theorem~\ref{thm-s} to all elements $x\in\be$. 
The obstacle is that finding a regular $a\in\g^*$ such that 
$\dim (\g^x)^{\bar a} = \rk\g$ and $(\bbk a+\bbk x)\cap\g^*_{\sf sing}\subset \bbk x$
 is a highly non-trivial task. 

\subsection{The r\^ole of nilpotent orbits}
Let $\N$ denote the set of nilpotent elements of $\g\simeq \g^*$. Any $G$-orbit in $\N$ is said to be nilpotent. As is 
well known, $\N/G$ is finite and any $G$-orbit in $\g$ can be contracted to a nilpotent one, see a construction below. 
This turns out to be extremely helpful in the theory of complete algebras. 

\begin{prop}        \label{lm-c-n}
Let $\eus A\subset \gS(\g)$ be a homogeneous subalgebra. 
If\/ $\trdeg (\eus A|_{Ge})=\frac{1}{2} \dim (Ge)$ for each nilpotent element $e\in \g^*$, 
then    $\trdeg (\eus A|_{Gx})=\frac{1}{2} \dim (Gx)$ for each $x\in\g^*$. 
\end{prop}
\begin{proof}
The statement is vacuous for nilpotent orbits. Assume therefore that $x\not\in\N$. Set 
$Y=\overline{\bbk^\times\! (G x)}$. This is a conical subvariety of $\g^*$ and $\dim Y=\dim Gx+1$.
By the method of {\it associated cones\/} introduced and developed 
in~\cite[\S\,3]{bokr}, there is an orbit $Ge\subset Y\cap\N$ such that  
$\dim(Ge)=\dim(Gx)$. Observe that $\textsl{d}_x \eus A=\textsl{d}_{tx} \eus A$
for each non-zero $t\in\bbk$, because $\eus A$ is homogeneous. 
Therefore 
\[
    \max_{y\in Y} \dim \textsl{d}_{y} \eus A = \max_{x'\in Gx}\dim \textsl{d}_{x'} \eus A  
\]
and in particular 
\begin{equation}    \label{eq-con1}
      \max_{x'\in Gx}\dim \textsl{d}_{x'} \eus A \ge \max_{e'\in Ge}\dim \textsl{d}_{e'} \eus A.  
\end{equation}
A possible way to conclude the proof would be to calculate $\dim(\textsl{d}_x \eus A \cap \g^x)$ 
and $\dim(\textsl{d}_e \eus A\cap\g^e)$. For instance, if $\eus A=\eus F_a$ is an MF-subalgebra with 
$a\in\g^*_{\sf reg}$, then  $\dim(\textsl{d}_y \eus A\cap\g^y)=l$ for any $y\in\g ^*$ by Lemma~\ref{comp-semi} and there is nothing else
to show.  But in case of a general $\eus A$, our approach is different. 

Since $x$ is not nilpotent, there is  a homogeneous non-constant polynomial $H\in\gS(\g)^{\g}$ such that $c=\deg H>0$ and $H(x)\ne 0$. Assume that homogeneous elements $\tilde a_1,\ldots,\tilde a_m\in \eus A$ are algebraically independent on $Ge$, but dependent on $Gx$. 
Without violating these assumptions, replace each $\tilde a_i$ with $a_i=\tilde a_i^c$. 
Set $c_i=\deg \tilde a_i$.
Let $\mathbf Q$ be a non-trivial relation among $a_i|_{Gx}$. Then 
$\mathbf Q\left(\dfrac{a_1}{H^{c_1}},\ldots,\dfrac{a_m}{H^{c_m}}\right)=0$
on $\bbk^{\times}(Gx)$.  Multiplying this equality by a suitable power of $H$ and restricting to
$Ge$, where $H$ vanishes, we obtain a non-trivial relation among ${a_1}|_{Ge},\ldots,{a_m}|_{Ge}$. 
A contradiction! Thus,
\begin{equation}            \label{eq-con2}
             \trdeg (\eus A|_{Ge})\le \trdeg (\eus A|_{Gx})
\end{equation}
and the result follows.  
\end{proof}

The {\it sheets\/} of $\g$ are the irreducible components of the locally closed subsets 
         $X^{(d)}=\{\xi\in\g\mid \dim(G\xi)=d\}$  for all $d$.
Let $G e$ be a nilpotent orbit in  $\overline{\bbk^\times\! (G x)}$ with $\dim(Ge)=\dim(Gx)$.
 Then $Ge$ is a nilpotent orbit in each sheet $S$ containing $Gx$. 
By a fundamental result of Borho and Kraft,
each sheet contains a unique nilpotent orbit \cite[Sect.~5.8.\,Kor.(a)]{bokr}. Therefore the associated 
cone of $Gx$, i.e., the variety $\overline{\bbk^\times\! (G x)}\setminus \bbk^\times\! (\ov{G x})$, 
is irreducible and the above-mentioned orbit $Ge$  is unique. 
Equation~\eqref{eq-con2} leads to the following statement. 

\begin{cl}               \label{cl-nilp}
Suppose that a homogeneous Poisson-commutative subalgebra $\eus A\subset \gS(\g)$ is complete on
a nilpotent orbit $Ge$. Then $\eus A$ is complete on any orbit $Gx$ such that $Gx$ and $Ge$ lie in one 
and the same sheet. 
\end{cl}

\begin{rmk} 
Proposition~\ref{lm-c-n} has a rather amusing application.
For, our considerations with nilpotent orbits easily recover the main result of a recent preprint~\cite{crooks}, which asserts that the MF-subalgebra 
$\eus F_a$ with $a\in\g^*_{\sf reg}$ is complete on each $Gx\subset\g^*_{\sf reg}$.  
Note that a more general result is already contained in Theorem~\ref{thm-s}, but the argument for the regular elements $x$ only can be made astonishingly simple and short. It uses neither Slodowy slices  nor the Kostant section. 
Namely, let $\{e,h,f\}$ be a principal $\tri$-triple in $\g$. 
Assume that $\eus F_a$ is not complete on $Gx$. Then 
$\eus F_a$ is not complete on $Ge$, see Proposition~\ref{lm-c-n}.  Then 
$\eus F_e$ is not complete on $Ga$ by Theorem~\ref{lm-dual}. Then $\eus F_e$ is not complete on $Ge$ again by Proposition~\ref{lm-c-n}.  However, this is absurd, since 
$\langle e,f\rangle_{\bbk}\subset \g^*_{\sf reg}\cup\{0\}$ and $\dim\textsl{d}_f \eus F_e = \bb(\g)$,
cf. Corollary~\ref{cl-MF}. 
\end{rmk}

In what follows, $e$ stands for an arbitrary nilpotent element of $\g$. 

\begin{prop}               \label{mf-generic}
There is a non-empty open subset $U\subset \g^*_{\sf reg}$ such that for any $a\in U$, the MF-subalgebra $\eus F_a$ is complete on 
{\sl every} adjoint orbit.
\end{prop}
\begin{proof}
Recall that $\N/G$ is finite. For each $Ge\subset\N$, the subset 
\[
     U(e)=\{a\in\g^*_{\rm reg}\mid \eus F_a \text{ is complete on } Ge\}
\]
is non-empty and open in $\g^*$ \cite[Thm\,3.2]{bols}. Let $U$ be the intersection of $U(e)$ taken 
over all nilpotent orbits. Then $U\ne\varnothing$ is open in $\g^*$. For any $a\in U$, 
the MF-subalgebra $\eus F_a$ is complete on every nilpotent and hence on every adjoint orbit, see 
Proposition~\ref{lm-c-n}.
\end{proof}

Proposition~\ref{mf-generic} opens ample possibilities for further generalisations.
It would be nice to prove that, for each $a\in\g^*_{\sf reg}$, $\eus F_a$ is complete on any adjoint orbit. 

\subsection{Complete families} 
\label{subs:complete}
For $a\in\g^*_{\rm reg}$, Theorem~\ref{thm-s} implies that 
$\trdeg \eus F_a=\bb(\g)$ and thereby the 
generators  $\partial_a^k H_i\in \eus F_a$ with $1\le i\le l$ and $0\le k <\deg H_i$  are 
algebraically independent. 
Suppose that $\dim \textsl{d}_x (\eus F_a)=\rk\g+\frac{1}{2}\dim(Gx)$ for some $x\in\g^*$. 
If $x$ is regular as well, one 
restricts  the polynomials $\partial_a^k H_i$ with $1\le i\le l$ and $0<k<\deg H_i$ to $Gx$ in order to obtain a complete family in involution.  
Suppose now that $\dim(Gx)<\dim\g-\rk\g$. Then some other generators of $\eus F_a$ 
become redundant on $Gx$. A natural question is, which ones?   There is a simple answer in 
types ${\sf A}$ and ${\sf C}$. 

Suppose that $\g$ is either $\gt{gl}_l$, $\gt{sl}_{l+1}$, or $\gt{sp}_{2l}$. 
As generating  symmetric invariants $H_1,\ldots,H_l$ we take coefficients of the 
characteristic polynomial. Assume that $\deg H_i>\deg H_j$ whenever $i>j$. 
Set $d_i=\deg H_i$. 
According to \cite[Sect.~2]{m-y}, $\eus F_x$ is a free algebra with a set 
$\{\partial_x^k H_i\mid 1\le i\le i, 0\le k \le s(i)\}$ of algebraically independent generators.
Moreover, the numbers $s(i)$ depend only on the partition of $e$, where $Ge$ is the dense orbit  
in the associated cone of $Gx$.   The dependence is very explicit, see \cite[Sect.~4]{m-y}. We note also that 
$\partial_e^{s(i)} H_i\in\gS(\g^e)$ and that $\partial_e^k H_i=0$ if $k>s(i)$.

\begin{prop}         \label{complete}
Suppose that $\g$ is of type {\sf A} or {\sf C}. Assume that $\eus F_a$ with $a\in\g^*_{\rm reg}$ is complete on $Ge$. Then the restrictions of $\partial_a^k H_i$ with 
$d_i>k > d_i-s(i)$ to $Gx$ is a  complete family in involution. 
\end{prop}
\begin{proof}
By virtue of Proposition~\ref{lm-c-n}, it suffices to prove the assertion  for $Ge$. 
According to~\eqref{ref-p}, the differential $\textsl{d}_e(\partial_a^k H_i)$ is equal to 
$\textsl{d}_a (\partial_e^{d_i-k-1} H_i)$ up to a non-zero rational scalar.   
If $d_i{-}k{-}1>s(i)$, then $\partial_e^{d_i-k-1} H_i=0$ and hence also $\textsl{d}_e(\partial_a^k H_i)=0$;  
if  $d_i{-}k{-}1=s(i)$, then $\textsl{d}_e(\partial_a^k H_i)\in \g^e$. 
The same statements for the differentials   hold at each point  $e'\in Ge$. 
If $F\in\gS(\g)$  and $\textsl{d}_{e'} F\in \g^{e'}$ for each $e'\in Ge$, 
then $F\vert_{Ge}$ is a constant,   if in addition $F$ is  homogeneous, then 
$F\vert_{Ge}=0$. 
Thus, the polynomials 
$\partial_a^k H_i$ with $k\le d_i-s(i)-1$ restrict to zero on $Ge$.  
The number of the remaining elements, $\partial_a^k H_i$ with 
$d_i>k > d_i-s(i)$,  is  equal to 
$\trdeg\eus F_e-l=\frac{1}{2}\dim(Ge)$, see \cite[Sect.~2]{m-y}.
\end{proof} 

\begin{ex}                \label{s(i)-A} 
Take $\g=\gln$. A nilpotent orbit $Ge\subset\g$ is determined by a
partition  $\boldsymbol{r}=(r_1,\ldots,r_t)$ of $n$, where $r_1\ge r_2\ge\ldots \ge r_t>0$ are the sizes of Jordan blocks of $e$. We then set $\co(\boldsymbol{r}):=Ge$.
The numbers $s(i)$ appeared in 
\cite[Thm~4.2]{trio} as the degrees of certain generators of $\gS(\g^e)^{\g^e}$, cf.~\cite[Lemma~1.5]{m-y}. They  are uniquely defined by the conditions  
\[
     \sum_{j=1}^{s(i)-1} r_j < i \le \sum_{j=1}^{s(i)} r_j. 
\]
To give a graphic presentation of the complete family of Proposition~\ref{complete}, 
we first arrange the polynomials $\partial_{a}^k H_i$ into the left justified Young tableau, where 
$H_n,\ldots,H_1$ form the first (top) row, $\partial_{a} H_n,\ldots,\partial_{a} H_2$ --- the second row, and so on until the last  (bottom)  row, where just $\partial_{a}^{n-1} H_n$ stands in the left column. The resulting diagram has consecutive rows of size $(n,n-1,\dots,1)$, hence it has 
$n(n+1)/2= \bb(\g)$ boxes.

Next, we define a certain colour pattern corresponding to $\co(\boldsymbol{r})$. This pattern is going to be used in Section~\ref{sec-A}. The recipe is the following: 
\begin{itemize}
\item[{\color{darkblue}$\diamond$}] in the top row paint the last (looking from the left) $r_1$ boxes in red and all boxes below them in green; 
\item[{\color{darkblue}$\diamond$}] in the second row find the  rightmost box that is not green, starting from it make 
a stripe of red boxes of length $r_2$, paint all the boxes below the stripe in  green; 
\item[{\color{darkblue}$\diamond$}]  if the first $m-1$ rows are painted and $r_m>0$, then 
find the rightmost box in the $m$-th row that is not green; starting from it make 
a stripe of red boxes of length $r_m$, and paint all the boxes below the stripe in green. 
\end{itemize}
The green boxes depict the complete family of Proposition~\ref{complete} 
and therefore there are $\frac{1}{2}\dim(Ge)$ of them. It is easily seen that we have $n$ red boxes.  
These boxes are going to be used in Section~\ref{sec-A}.
\\ \indent
The colour patterns corresponding to the partitions
 $(3,2,1)$, $(4,1)$, and $(2,2,2,1)$ are presented below. 

\vskip1.5ex
\ytableausetup{mathmode, boxsize=1.2em}
\quad
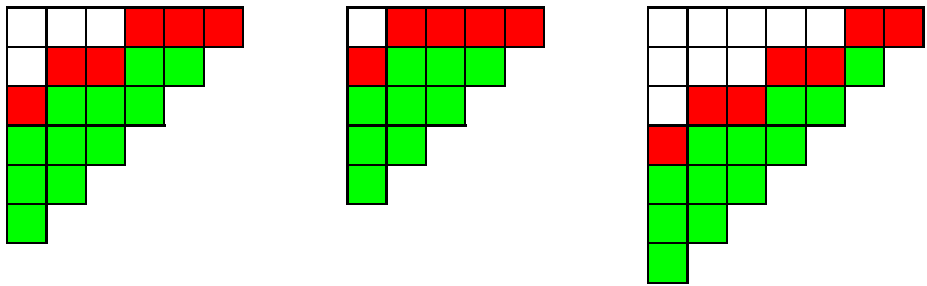
\begin{ytableau}
*(white)& *(white)& *(white)& *(red) & *(red)& *(red)\\
*(white)& *(red) & *(red) & *(green) &*(green) \\
*(red) & *(green) &*(green) &*(green) \\
*(green) &*(green) &*(green) \\
*(green) &*(green) \\
*(green) \\
\end{ytableau} \qquad\quad
\begin{ytableau}
*(white)& *(red) & *(red)& *(red) & *(red)\\
*(red) & *(green) &*(green) &*(green) \\
*(green) &*(green) &*(green) \\
*(green) &*(green) \\
*(green) \\
\end{ytableau}
\qquad\quad
\begin{ytableau}
*(white)&*(white)&*(white)& *(white)& *(white)& *(red) & *(red)\\
*(white)& *(white)& *(white)& *(red) & *(red) & *(green)\\
*(white)& *(red) & *(red) & *(green) & *(green)\\
*(red) & *(green) &*(green) &*(green) \\
*(green) &*(green) &*(green) \\
*(green) &*(green) \\
*(green) \\
\end{ytableau}
\end{ex}

\section{Flag varieties and coisotropic actions } 
\label{sec-flag}

\noindent
Suppose for a while that $G$ is a {\bf complex} reductive group. 
Let $B\subset G$ be a Borel subgroup, 
$T(\mathbb C)\subset B$  a maximal torus in $G$, $P\subset G$ 
a parabolic containing $B$. Fix also a maximal compact subgroup $K\subset G$ such that 
$T=K\cap T(\mathbb C)$ is a maximal torus in $K$. 
Set $\gt k=\Lie K$, $\gt t=\Lie T$. Let further $V_\lb$ be a finite-dimensional 
simple $G$-module with a highest weight vector $v_\lambda$. 
Standard facts are that $G/B \simeq K/T$ and $G/P\simeq K/L$, where 
$L=P\cap K$, and the (real) symplectic structure on 
$G/P =G\langle v_\lambda\rangle \subset \mathbb PV_\lambda$ is the same as on the (co)adjoint orbit 
$K \lambda \subset \gt k^*$. This is one of the reasons, why {\it integrable systems} ($\sim$ complete families in involution) on adjoint  orbits of compact groups are of particular interest.  

Definition~\ref{com-fam} can be reformulated  for any symplectic manifold or variety $M$. 
If $M$ is not algebraic, then one has to consider smooth (or differentiable) functions and 
replace ``algebraically independent" with ``functionally independent". 
In what follows, we write simply ``a complete family" instead of ``a complete family in involution". 
Strictly speaking, an integrable system includes also a choice of a {\it Hamiltonian}, a function 
$\boldsymbol{H}$ on $M$ that Poisson-commutes with a complete family. Fortunately,
an arbitrary element of a complete family can be chosen as $\boldsymbol{H}$. 

The most famous example of a complete family on a flag variety is the Gelfand--Tsetlin 
system of  Guillemin--Sternberg in the ${\rm U}_n$-case~\cite{gs1},
the $\bl$-system in our terminology, 
see the Introduction and  Section~\ref{sec-A} for its description. 
There is also a direct analogue in the orthogonal case \cite{gs2} and a  symplectic 
variation due to Harada~\cite{Meg}. We demonstrate below that MF-subalgebras 
lead to integrable systems on flag varieties. Our construction is independent of the type of $G$. 

Although we have assumed  so far that $\overline\bbk=\bbk$,  
MF-subalgebras can be defined in the same way 
over $\mathbb Q$ for the rational forms of $\g$, as well as for the real forms. 
In particular, the method works for $\gt k$. This was already clear to Mishchenko and Fomenko \cite{mf}. 
Observe that $\gS(\gt k)^{\gt k}\otimes_{\mathbb R}\mathbb C=\gS(\g)^{\g}$.
Choose 
a parameter $a\in\ka^*$ and let $\eus F_a\subset \gS(\gt k)$ be the MF-subalgebra associated with $a$. 
Then $\eus F_a(\mathbb C)=\eus F_a\otimes_{\mathbb R}\mathbb C$ is the complex MF-subalgebra of 
$\gS(\g)$ associated with $a$, where $a$ is regarded as a complex valued linear function on $\g$.

Let $\{e,h,f\}\subset \g$ be a principal $\gt{sl}_2$-triple such that
\[
 \left<e,h,f\right>_{\BC}\cap \gt k = \left<ih, f-e, if+ie\right>_{\BR} \ \text{ and } \ ih\in\gt t .
\]
Note that  $\g^x=\ka^x\otimes_{\BR}\BC$ for any $x\in\ka^*$. Hence 
$\ka^*_{\sf reg}\subset \g^*_{\sf reg}$.

\begin{prop}     \label{MF-flag}
Take $a\in\gt k^*_{\sf reg}$. Then the real MF-subalgebra $\eus F_a$  is complete on  
any orbit $Kx\in\gt k^*$  and therefore on any  flag variety $G/P=G\langle v_\lambda\rangle$. 
If we choose $a=f-e\in\gt k\simeq \gt k^*$, then 
$\dim\textsl{d}_x ({\eus F_a}\vert_{Kx})=\frac{1}{2}\dim(Kx)$ for every $x\in\gt t^*$.
\end{prop}
\begin{proof}
All elements of $\ka$ are semisimple. By Theorem~\ref{thm-s}, $\eus F_{a}(\mathbb C)$ is complete on 
$G x\subset\g^*$ if $a\in\gt k^*_{\sf reg}$ and
$x\in\gt k^*$.  The equality $\dim\textsl{d}_y ({\eus F_a(\mathbb C)}\vert_{Gx})=\frac{1}{2}\dim(Gx)$ holds 
for each $y\in U$, where $U\subset Gx$ is a non-empty Zariski open subset.  In the complex Zariski 
topology,  $Kx$ is dense in $Gx$. Hence $U\cap Kx\ne\varnothing$.  
 By a standard linear algebra argument, for any $x\in\gt k^*$,
we have  $\textsl{d}_x \eus F_a\otimes_{\mathbb R}\mathbb C =\textsl{d}_x \eus F_a(\mathbb C)$. Thus, 
$\eus F_a$ is complete on $Kx$. 

If $a=f-e$ and $x\in\gt t^*$, then   $\dim\textsl{d}_{x} (\eus F_a(\BC)\vert_{Gx})=\frac{1}{2}\dim(Gx)$
according to the proof of Theorem~\ref{thm-s}.  
Hence here $\dim_{\mathbb R}\textsl{d}_{x} (\eus F_a \vert_{K x})=\frac{1}{2}\dim_{\mathbb R}(K x)$.
\end{proof}

The Gelfand--Tsetlin  system of Guillemin--Sternberg 
is complete on each adjoint orbit of ${\rm U}_n$.  
The key point here is that the action of ${\rm U}_{n-1}$ on a (co)adjoint orbit of 
${\rm U}_n$ is {\it coisotropic}, which is formulated in~\cite{gs2}. 
Guillemin and Sternberg prove this assertion  if the orbit in question is regular, 
the non-regular case being illustrated through examples. The statement, for both ${\rm U}_n$ 
and ${\rm SO}_n(\mathbb R)$, is attributed to Heckman \cite{Hek}, see e.g.~\cite[p.\,225]{gs2}. 
Below, we give a modern perspective on the matter  and show that the non-regular case follows easily from the regular one.   

\subsection{Coisotropic actions} 
\label{subsec-cois} 
The  symplectic manifolds (or varieties) $(M,\omega)$ endowed with a coisotropic action of a group are 
also known as  the ``multiplicity-free spaces" \cite{gs-free,Alan}. The starting point is a {\it Hamiltonian 
action} of a group $Q$ on $M$, see e.g.~\cite[Sect.\,2]{gs-moment} for the definition. In this section, we
assume that either $M$ is a smooth variety over $\bbk$ and $Q$ is an affine algebraic group defined 
over $\bbk$  or $M$ is a homogeneous space of a compact real group $K$ and  
$Q$ is a compact real group. In both cases, $M$ is assumed to be irreducible.  

Associated with the Hamiltonian action of $Q$, there is a {\it moment map} $\mu=\mu_Q\!: M\to \q^*$, 
see~\cite[Sect.\,3]{gs-moment}. In this paper, we are interested only in cases, where the moment 
map is defined globally. 
The elements of $\mu^*(\gS(\q))$ are functions on $M$ and they are called either 
{\it Noether integrals} or {\it collective functions}.   
We have either $\mu^*(\gS(\q))\subset \bbk[M]$ or $\mu^*(\gS(\q))\subset \mathbb R[M]$, depending on the context.  The name ``Noether integrals" is justified by the following theorem of 
Emmy  Noether:  $\{F,\mu^*({\gS}(\q))\}=0$ for each $Q$-invariant function $F$ on $M$.
The term ``collective functions" is introduced in~\cite{gs2}. 

Let $\bL$ denote either  $\bbk$ or $\mathbb R$. Write $\bL(M)^Q$ for the field of $Q$-invariant rational functions on $M$.
For $x\in M$, set $\gt q x=T_x(Qx)$. 

\begin{df}          \label{df-cois}
A Hamiltonian action of $Q$ on $M$ is {\it coisotropic} if
$(\q x)^{\perp}\subset (\q x)$ for generic $x\in M$, where the orthogonal complement is taken w.r.t.
the symplectic form $\omega_x$. 
\end{df}

Since  $\omega_x$ is non-degenerate, the condition $(\q x)^{\perp}\subset (\q x)$ is equivalent to that
\begin{equation}      \label{cois-2}
       \omega_x \ \text{ vanishes on } \ (\q x)^{\perp}.
\end{equation} 
There are many  equivalent conditions that define coisotropic actions, see e.g.~\cite[Sect.\,2]{gs2}.
Some of them are presented below. 
 
The Poisson structure $\pi$ on $M$ is given by $\pi(x)=(\omega_x^{-1})^t$ at $x\in M$. 
Here $\omega_x$ is a skew-symmetric form on $T_xM$ and $\pi(x)$ is a skew-symmetric
form on $T^*_x M$. 
By duality between $\omega$ and $\pi$, we have 
\begin{equation}\label{cois-3}
{\omega_x}\vert_{(\gt q x)^{\perp}}=0 \ \Longleftrightarrow \ \pi(x)\vert_{\Ann(\gt q x)}=0.
\end{equation}  
Let $F$ be a $Q$-invariant rational function on $M$ such  that $\textsl{d}_x F$ is defined. 
Then $\textsl{d}_x F$ vanishes on $\q x$, i.e., $\textsl{d}_x F\in\Ann(\gt q x)$. 
By the Rosenlicht theorem, see e.g.~\cite[Thm~2.3]{VP}, the
rational $Q$-invariants on $M$ separate generic $Q$-orbits. Hence there is a non-empty subset 
$U\subset M$ such that for each $y\in U$ there are rational functions  
$F_1,\ldots,F_m\in\bL(M)^Q$ satisfying $\langle F_i \mid 1\le i\le m\rangle_{\bL}=\Ann(\gt q y)$. 
Therefore  \eqref{cois-2} holds generically if and only if  
\begin{equation}        \label{cois-4}
    \bL(M)^Q  \ \text{ is Poisson commutative.} 
\end{equation} 
If $\trdeg\bL[M]^Q=\trdeg\bL(M)^Q$, then \eqref{cois-4} is equivalent to 
\begin{equation}        \label{cois-5}
    \bL[M]^Q  \ \text{ is Poisson commutative.} 
\end{equation} 
Note that in the compact setting, the regular invariants $\mathbb R[M]^Q$ separate 
all $Q$-orbits. Further conditions involve $\mu$. 
 
Observe that \ $\ker(\textsl{d}_x\mu) =(\q x)^\perp$, see e.g.~\cite[Eq.~(1.6)]{gs2} or  
\cite[Eq.~(56)]{vin-com}.  Thus, 
\beq      \label{cois-6}
    (\q x)^\perp \subset \q x \ \Longleftrightarrow \   (\textsl{d}_x\mu)^{-1}(\q \mu(x))=\q x. 
\eeq 
Suppose that $\dim M=2n$ and $F_1,\ldots,F_n$ is a complete family on $M$ 
consisting of Noether integrals, i.e., $F_i\in {\rm Im\,} \mu^*$ for each $i$. 
For $x\in M$, set $L(x)=\langle\textsl{d}_x F_i\mid 1\le i\le n\rangle_{\bL}$.
Then ${\pi(x)}\vert_{L(x)}=0$ and $L(x)$ is ortogonal to $\textsl{d}_x(\bL(M)^Q)$ w.r.t. $\pi(x)$. 
If $x$ is generic, then $L(x)$ is a Lagrangian subspace of $T_x^* M$ w.r.t. $\pi(x)$ and 
$\textsl{d}_x(\bL(M)^Q)=\Ann(\gt q x)$.  For such an $x$, we have $\Ann(\q x)\subset L(x)$ and hence 
$\pi(x)$ vanishes on $\Ann(\gt q x)$. Therefore, it follows from~\eqref{cois-3}, see also the theorem 
in~\cite[Sect.\,2]{gs2}, that the following assertion is true:
\begin{itemize}
\item[{\sf(NF)}]
there is a complete family on $M$ consisting of Noether integrals 
only if the action of $Q$ on $M$ is coisotropic. 
\end{itemize}

\begin{thm}[\cite{gs2}]        \label{thm-cois1}
The action of\/ ${\rm U}_{n-1}$ on any adjoint orbit of\/ ${\rm U}_n$ is coisotropic.
\end{thm} 
\begin{proof}
Set $Q={\rm U}_{n-1}$, $M={\rm U}_n x\subset \gt u_n$. 
Suppose first that $M$ is a regular ${\rm U}_n$-orbit. Take $y\in \mu(M)$. Then $y$ is a regular 
point of  $\q$~\cite[Sect.\,4]{gs2} and 
$Q_y$ acts on $\mu^{-1}(y)$ transitively if $y\in \mu(M)$ is generic~\cite[Eq.\,(2.5)]{gs2}. 
Combining these facts with~\eqref{cois-6}, we obtain that 
\eqref{cois-2} holds at each point $x\in\mu^{-1}(y)$.  It is also true that  
$\{\gS(\gt u_n)^Q,\gS(\gt u_n)^Q\}$ vanishes on $M$, cf. \eqref{cois-5}. Since this holds for 
any regular orbit, $\gS(\gt u_n)^Q$ is Poisson-commutative. 

Next, let $M\subset \gt u_n$ be an arbitrary adjoint orbit. 
Since $Q$  is compact, $$\BR(M)^Q={\rm Quot}(\BR[M]^Q)$$ and $\BR[M]^Q$ is the restriction of
$\gS(\gt u_n)^Q$ to $M$. In particular, $\BR[M]^Q$ is Poisson-commutative and this implies 
that \eqref{cois-2} holds for genetic $x\in M$. 
\end{proof}

Theorem~\ref{thm-cois1} combined with an inductive argument of~\cite[(2.9)]{gs2}, 
yields the following assertion.  

\begin{cl}[\cite{gs2}]
The integrable system of\/ \cite{gs1},
the type {\sf A} $\bl$-system in our terminology,  is complete on any adjoint orbit of\/ ${\rm U}_n$.  
\end{cl}

A similar inductive argument applies in the orthogonal case, too. Actually, Section~\ref{sub-strong} 
contains a thorough discussion of the fact that the action of  $\SO_{n-1}(\mathbb R)$ on every 
adjoint orbit of $\SO_n(\mathbb R)$ is coisotropic.  

The ``multiplicity-free spaces" of~\cite{gs-free,Alan} are related to multiplicity-free decompositions
and {\it spherical varieties}. An  algebraic $\bbk$-variety $X$ acted upon by a reductive group $G$ 
is said to be {\it spherical}, if a Borel subgroup $B\subset G$ acts on $X$ with an open orbit.  

Suppose that $M$ is K\"ahler and $Q$ is a compact real group. Then the action of $Q$ on $M$ is 
coisotropic if and only if 
\begin{itemize}
  \item[{\sf (SC)}] \ $M$ is a spherical $Q(\mathbb C)$-variety \cite[Sect.~6]{Alan}. 
\end{itemize}
A  complex flag variety $G/P$ is definitely K\"ahler.  Take $Q\subset  K\subset G$. 
If $G/B$ is spherical w.r.t. $Q(\mathbb C)$, then $G/P$ is also a spherical 
$Q(\mathbb C)$-variety for each parabolic $P$.  This is another way to see that if a generic
adjoint orbit of  $K$ is coisotropic w.r.t. $Q$, then each adjoint orbit of  $K$ is also $Q$-coisotropic.

\subsection{Strong Gelfand pairs}        \label{sub-strong}
Among pairs of reductive groups $H\subsetneq G$, two  occupy the most prominent position. 
These are the {\it strong Gelfand pairs} $(\GL_n(\bbk),\GL_{n-1}(\bbk))$ and 
$(\SO_n(\bbk),\SO_{n-1}(\bbk))$.  Up to local isomorphisms, products, products with $(H,H)$, and pairs 
$(\bbk^{\times}, \{e\})$,  these are the only strong Gelfand pairs, see \cite{Kr} and~\cite[Sect.~4]{Hek}. 

Strong Gelfand pairs can be characterised by a host of equivalent conditions. Below we present a selection of these conditions:
\begin{itemize}
\item[{\sf(Sph1)}] \ the homogeneous space 
$(G\times H)/H$ is a spherical $(G\times H)$-variety; 
\item[{\sf(Sph2)}] \ $G/B$ is a spherical $H$-variety; 
\item[{\sf(Br)}] \ each irreducible finite-dimensional representation $V_\lambda$ of 
$G$ decomposes without multiplicities under the  action of $H$;
\item[{\sf(Com)}] \ the algebra $\eus U(\g)^{\gt h}$ is  commutative; 
\item[{\sf(PCm)}] \ the algebra $\gS(\g)^{\gt h}$ is Poisson-commutative; 
\item[{\sf(Cois)}] \ the action of $H$ on each closed orbit $Gx\subset \g^*$ is coisotropic; 
\item[{\sf(DCn)}] \ $\gS(\g)^{\gt h}=\mathsf{alg}\langle \gS(\g)^{\g},\gS(\gt h)^{\gt h}   \rangle$; 
\item[{\sf(CtB)}] \ the action of $H$ on $T^*(G/P)$ is coisotropic for each parabolic $P\subset G$. 
\end{itemize}

It is a classical fact that the pairs $(\GL_n(\bbk),\GL_{n-1}(\bbk))$ and $(\SO_n(\bbk),\SO_{n-1}(\bbk))$ 
satisfy {\sf (Br)}. It took a long time and many papers to prove the equivalences of the above conditions. 
Below is a brief outline.  

\begin{rmk}    \label{SGP-cond}
[Arguments for the equivalences.] 
The fact that  {\sf(Sph1)} $\Leftrightarrow$ {\sf(Sph2)} is  observed in 
\cite{DA}, see Eq.~(5) on page 26 therein.   

Both equivalences  {\sf(Sph1)} $\Leftrightarrow$ {\sf(Br)} and   {\sf(Sph2)} $\Leftrightarrow$ {\sf(Br)}
are results of \cite{vin-k}. The action of $H$ on the flag variety 
$G\langle v_\lambda \rangle\subset \mathbb P V_\lambda$ is spherical if and only if 
each $V_{n\lambda}$ with $n\in\mathbb N$ decomposes without multiplicities under the action of $H$.
In the affine case, $(G\times H)/H$ is a spherical $(G\times H)$-variety if and only if 
$\dim(V_\lambda\otimes V_\mu)^H\le 1$ for all irreducible finite dimensional $G$-modules $V_\lambda$ 
and $H$-modules $V_\mu$.   

A simple proof for the equivalence {\sf(Br)} $\Leftrightarrow$ {\sf(Com)} is given in~\cite{john}.  

Since $\gS(\g)^{\gt h}=\gr(\eus U(\g)^{\gt h})$, we have {\sf(Com)} $\Rightarrow$ {\sf(PCm)}.  

The implication {\sf(PCm)} $\Rightarrow$ {\sf(Sph1)} can be extracted from 
the proof of \cite[Satz\,2.3]{Fr}, see the implication ($2' \Rightarrow 3$) therein. 
In \cite[Satz\,2.3]{Fr}, it is shown that  {\sf(PCm)} $\Leftrightarrow$ {\sf(Cois)} $\Leftrightarrow$ {\sf(DCn)}. That proof exploits the classification of strong Gelfand pairs. 
Below we give an alternative, classification-free argument, 
see Theorem~\ref{equiv-new}.   

Observe that the implication  {\sf(DCn)} $\Rightarrow$ {\sf(Com)} is almost trivial. 
Let $\varpi\!:\gS(\g)\to \eus U(\g)$ be the symmetrisation map. It is a homomorphism of $G$-modules. 
Thereby $\eus{U}(\g)^{\gt h}=\varpi(\gS(\g)^{\gt h})$. Suppose that {\sf(DCn)} holds. Then
$\gS(\g)^{\gt h}$ is generated by $\gS(\gt h)^{\gt h}$ as an $\gS(\g)^{\g}$-module.  
Therefore $\eus U(\g)^{\gt h}$ is generated by $\eus U(\gt h)^{\gt h}$ as a 
$\eus U(\g)^{\g}$-module.   
Since 
$[\eus U(\gt h)^{\gt h},\eus U(\gt h)^{\gt h}]=0$, the condition  {\sf(Com)}  holds. 

Finally, the equivalence {\sf(CtB)} $\Leftrightarrow$ {\sf(Sph2)} follows from \cite[Satz~7.1]{Fr-WM}, see 
also \cite[Chapter 2,\,\S 3]{vin-com} and in particular Theorem~2 therein.  
\end{rmk}

An open subset $U$ of an irreducible algebraic variety $X$ is said to be {\it big} if 
$\dim X\setminus U\le \dim X-2$. 

\begin{lm}             \label{lm-double} 
Let $H\subset G$ be a reductive subgroup. 
Set\/ $\eus C_1=\mathsf{alg}\langle \gS(\g)^{\g},\gS(\gt h)^{\gt h}   \rangle$. 
Then $\eus C_1$ is an algebraically closed subalgebra of\/ $\gS(\g)$.  
\end{lm}
\begin{proof}
If $\gt h$ contains a non-trivial ideal of $\g$, we can replace $H$ by a smaller subgroup without altering 
$\eus C_1$. Therefore assume that 
$\gt h$ contains no non-trivial ideals of $\g$. Then  
$\eus C_1$ is generated by homogenous algebraically independent 
elements $\{\boldsymbol{c}_1,\ldots,\boldsymbol{c}_r\}$ such that 
$\gS(\g)^{\g}=\bbk[\boldsymbol{c}_1,\ldots,\boldsymbol{c}_l]$ and
$\gS(\gt h)^{\gt h}=\bbk[\boldsymbol{c}_{l+1},\ldots,\boldsymbol{c}_r]$, see \cite[Satz~2.1]{Fr}.  
For $x\in\g^*$, set $\bar x=x\vert_{\gt h}$.
In view of the Kostant regularity criterion \eqref{eq:ko-re-cr}, we have 
$\dim\textsl{d}_x \eus C_1=r$ if and only if 
\[
     x\in\g^*_{\sf reg}, \  \bar x\in\gt h^*_{\sf reg}, \ \text{ and } \  \g^x\cap\gt h^{\bar x}=0.
\]
The first two conditions hold on big open subsets. The third one holds if and only if $\h^x=0$. 
Write $x=\bar x+y$ with $y(\gt h)=0$. Then 
$\h^x=(\h^{\bar x})^{y}$. Our goal is to show that the third condition is also satisfied on a big open subset. 

Let $\gt h^*_{\sf sreg}\subset\gt h^*_{\sf reg}$ be the subset of regular semisimple elements. 
If $\bar x\in\gt h^*_{\sf sreg}$, then the stabiliser $H^{\bar x}$ is a torus. Since the 
action of $H^{\bar x}$ on $\Ann(\gt h)\subset\g^*$ is self-dual, 
$(\gt h^{\bar x})^{y}=0$ on a big open subset of $\Ann(\gt h)$. 
Assume that $D\subset \g^*$ is an irreducible divisor such that 
$\gt h^x\ne 0$ for each $x\in D$. Choose an $H$-stable decomposition   
$\g^*=\gt h^*\oplus\Ann(\gt h)$. Let $p_1$ and $p_2$ be the projections on the first and the 
second summands, respectively. The above argument shows that 
$p_1(D)$ is contained in $\gt h^*\setminus \gt h^*_{\sf sreg}$ and hence necessary 
$\overline{p_2(D)}=\Ann(\gt h)$. 
Now let $y\in\Ann(\gt h)$ be a generic point. 
Then $H^y$ is  a reductive subgroup of $H$.  
Arguing by induction on $\dim\g$ we show that 
$(\gt h^y)^{x'}=0$ for all $x'$ from a big open subset of $\gt h^*$. 
Hence there is no $D$ as above.  

Taking the intersection of three big open subsets, we conclude that the differentials 
$\textsl{d}\boldsymbol{c}_1,\ldots,\textsl{d}\boldsymbol{c}_r$ are linearly independent on a big open 
subset. Since each $\boldsymbol{c}_i$ is homogeneous, \cite[Thm~1.1]{trio} applies and 
guarantees that $\eus C_1$ is algebraically closed.   
\end{proof}

\begin{thm}[cf.~{\cite[Satz\,2.3]{Fr}}]       \label{equiv-new} 
The conditions {\sf(PCm)}, {\sf(Cois)}, and  {\sf(DCn)} are equivalent.   
\end{thm}
\begin{proof}
For any closed orbit $Gx\subset \g^*$, generic $H$-orbits in $Gx$ are closed as well~\cite{Luna}. 
Hence they are separated by regular $H$-invariants and $\bbk(Gx)^H$ is the quotient field of 
$\bbk[Gx]^H$. As $H$ is reductive, $\bbk[Gx]^H$ is the restriction of $\bbk[\g^*]^H$ to $Gx$. 
Thus, {\sf(PCm)} $\Rightarrow$ {\sf(Cois)}. 

Since $\gS(\g)^{\g}$ is the Poisson centre 
of $\gS(\g)$ and $\gS(\gt h)^{\gt h}$ is Poisson-commutative, we have {\sf(DCn)} $\Rightarrow$ {\sf(PCm)}. 

It remains to show that   {\sf(Cois)} $\Rightarrow$ {\sf(DCn)}. 
Suppose that  {\sf(Cois)} holds. 
One of the equivalent interpretations, see \eqref{cois-6}, implies that 
$\trdeg \bbk(Gx)^H \le \rk\gt h$ for each $Gx\subset\g^*$. 
Thereby $\trdeg \gS(\g)^{\gt h}\le \rk\g+\rk\gt h$. 
We may safely assume that $\gt h$ contains no proper ideals of $\g$.  
By \cite[Satz~2.1]{Fr}, $\trdeg \eus C_1=\rk\g+\rk\gt h$ for $\eus C_1$ as in 
Lemma~\ref{lm-double}. 
Clearly $\eus C_1\subset \gS(\g)^{\gt h}$ is an algebraic extension. 
Since $\eus C_1$ is algebraically closed by Lemma~\ref{lm-double}, we have 
 $\eus C_1 = \gS(\g)^{\gt h}$ and  {\sf(DCn)} holds.
\end{proof}

\subsection{Cotangent bundles and Richardson orbits}  
\label{subs:CtB} 
There are similar results for nilpotent orbits, where a different kind of invariant theory is involved. 

Let now $H\subset G$ be an arbitrary  reductive subgroup of a reductive group $G$. Take a parabolic 
$P\subset G$. 
Then the action of $H$ on $G/P$ is spherical if and only if the action of $H$ on $T^*(G/P)$ is coisotropic, see \cite[Satz\,7.1]{Fr-WM} and also \cite[Chapter\,2,\,\S 3]{vin-com}.  
The image of the moment map 
\[
     \mu : T^*(G/P)\to \g^*
\]
 is isomorphic to $G {\gt u}$, where $\gt u\subset\gt p=\Lie P$ is
the nilpotent radical of $\gt p$.  Let $e\in\gt u$ be a {\it Richardson element}, which means that  
$\co=Ge$ is dense in $G\gt u$.  Comparing the symplectic structures on $T^*(G/P)$ and  on 
$\co$, one obtains the following result.

\begin{thm}[{\cite[Thm\,2.6]{av}}]        \label{sp-R}
The action of $H$ on $\mathcal O$ is coisotropic if and only if
$G/P$ is a spherical $H$-variety.  \qed
\end{thm}

For a strong Gelfand pair $(G,H)$, this implies that the $H$-action on any Richardson $G$-orbit is
coisotropic. Since every nilpotent orbit in $\gt{gl}_n$ is Richardson, 
\beq              \label{eq-cois}
     \text{the
 $\GL_{n-1}(\bbk)$-action  on any nilpotent adjoint orbit of $\GL_n(\bbk)$ is coisotropic.}
\eeq
Coisotropic actions of subgroups $Q\subset G$ on adjoint orbits of 
a semisimple group  $G$ have also been studied in \cite{zor}.

\section{The polynomial Gelfand--Tsetlin integrable system in type {\sf A}} 
\label{sec-A}

\noindent
In this section, $\g=\gln=\gln(\bbk)$. 
Let $\{E_{ij} \}_{i,j=1}^n\in\gln$ be the matrix units. 
Fix the chain of subalgebras
\beq     \label{eq:chain}
     \gln \supset \gt{gl}_{n-1}\supset \ldots \supset \gt{gl}_2\supset \gt{gl}_1 ,
\eeq
where $\gt{gl}_{n-k}=\left<E_{ij} \mid i,j > k\right>_{\bbk}$. In other words, let us fix a basis
$\{v_1,\dots,v_n\}$ for $V_n=\bbk^n$ and set $V_j=\left<v_{n-j+1}, \dots, v_n\right>_{\bbk}$.
Then $V_n\supset\dots \supset V_1$ is a full flag and $\gt{gl}_j=\gt{gl}(V_j)$ for all $j$.
For any matrix $A\in\gln$, let $A_m$ denote 
the south-east corner of $A$ of size $n-m$, i.e., $A_m\in \gt{gl}_{n-m}$.
For each $m\in\{0,1,\dots,n-1\}$, let $\{\Delta_k^{[m]}\mid 1\le k\le n-m\}$  be the coefficients of the characteristic polynomials of $A_m$. Here
$\Delta_k^{[m]}\in \eus S^k(\gt{gl}_{n-m})\subset \eus S^k(\gln)$, and we also write $\Delta_k=\Delta_k^{[0]}$.

The {\it Gelfand--Tsetlin (=\,GT) subalgebra\/} $\eus C\subset\gS(\gln)$
is generated by  
\[
  \Delta_1,\ldots,\Delta_n, \Delta_1^{[1]},\ldots, \Delta_{n-1}^{[1]},\Delta_1^{[2]}, \ldots, 
  \Delta_k^{[m]}, \ldots, \Delta_1^{[n-1]}.
\]
Note that $\eus C=\gr(\widetilde{\eus C})$, where $\widetilde{\eus C}\subset\eus U(\gln)$ is
the commutative subalgebra defined and studied by Gelfand and Tsetlin~\cite{gt-1}. Therefore,
these generators are algebraically independent, $\trdeg\eus C=\bb(\gln)$, and
$\{\eus C,\eus C\}=0$.   

By~\cite{t:max}, $\eus C$ is  a maximal Poisson-commutative subalgebra of $\gS(\gln)$. The same result 
is independently obtained in~\cite[Thm~3.25]{KoW}. Kostant and Wallach also prove that $\eus C$ 
is complete on every regular orbit, see Theorem~3.36 in loc.\,cit. 
\\ \indent
We prove below that $\eus C$ is complete on {\bf every} (co)adjoint orbit. 

\begin{df}     \label{def:str-nilp}
A matrix $A\in\gln$ is said to be
\begin{itemize}
\item[\sf (i)]  \ {\it strongly regular}, if $\dim\textsl{d}_A\eus C=\bb(\gln)$;
\item[\sf (ii)] \ {\it strongly nilpotent}, if $\Delta_k^{[m]}(A_m)=\Delta_k^{[m]}(A)=0$ for
$0\le m\le n-1$ and $1\le k\le n-m$.
\end{itemize}
\end{df}

\begin{thm}      \label{strong}
Any nilpotent orbit $\co\subset\g^*$ contains a strongly nilpotent element $e\in\co$ such that 
\\[.7ex]  
\centerline{$ \dim \textsl{d}_e \eus C = n+\frac{1}{2}\dim \co$ \ 
and \ $\dim (\textsl{d}_e \eus C\cap \g^e)=n$.}
\\[.7ex] 
In particular, $\eus C$ is complete on $\co$.
\end{thm}
\begin{proof}
As above, for $e\in\gln\simeq\gln^*$, let $e_m\in\gt{gl}_{n-m}$ denote the corresponding 
south-east corner, where $0\le m< n$. In particular, $e_0=e$ and $e_{n-1}\in\gt{gl}_1$. If all $\{e_m\}$ 
are nilpotent, then $\textsl{d}_e \Delta_k^{[m]}=(e_m)^{k-1}$ as a matrix.   

Let $\co=\co(\boldsymbol{r})$, where $\boldsymbol{r}=(r_1,r_2,\dots, r_t)$ is the corresponding
partition of $n$. If $t=1$, i.e., $r_1=n$,  then $\co=\co(n)$ is regular and a Jordan normal form adapted to the chain~\eqref{eq:chain} provides a strongly nilpotent element in $\co$. Namely, take a basis  
$\{v_1,\dots,v_n\}$ for $\bbk^n$ as above and set $e v_j=v_{j+1}$ for all $j$. (Here and below we assume that $v_j=0$ for $j>n$.) In this case,
$\textsl{d}_e\eus C=\be$, the unique Borel subalgebra containing $e$, and the assertions are clear.
Therefore, we always assume below that $t\ge2$, i.e., $r_2>0$.

Let $e'\in\gt{gl}_{n-1}$ be a nilpotent element defined by the partition 
$\boldsymbol{r}_1=(r_1+r_2-1,r_3,\ldots,r_t)$. As the next step
we will construct a representative $e\in\co$ such that $e'=e_1$. 
Our construction will not affect the Jordan blocks for $r_3,\dots,r_m$. 
Set $c=r_1+r_2$.

Let $\{v_2,\ldots,v_c\}$ be a Jordan basis for the first block of $e'$, i.e., $e'v_j=v_{j+1}$ for
$2\le j\le c-1$  and $e'v_c=0$. 
Define $e\in\gln$ as follows:
\\[.4ex]   \centerline{
$e v_1 = - v_{r_1+2}$, \ $e v_{r_1} = v_{r_1+1} + v_1$, \
$e v_c=0$, and $e v_j=e' v_j=v _{j+1}$ for
$j\ne 1, r_1, c$.}
\\[.5ex]
Then $\{v_2,\ldots,v_{r_1}, v_{r_1+1} {+} v_1\}$ is a Jordan basis for the block of size $r_1$ for $e$
and if $r_2\ge 2$, then $\{\frac{1}{2}(v_{r_1+1} {-} v_1),v_{r_1+2},\ldots,v_{c}\}$ is a Jordan basis 
for the block of size $r_2$ for $e$.
For $r_2=1$, the second block consists of $v_{r_1+1}-v_1$ or just $v_1$.

In Example~\ref{s(i)-A}, we have constructed the colour pattern
associated with  $\co(\boldsymbol{r})$.  For further considerations, replace 
each $\partial^m_{a} H_k$ in that pattern with $\Delta^{[m]}_{k-m}$. 

In order to prove the theorem, we  argue by induction on $n$. The case $n=1$ is void. 
By the  inductive hypothesis,  both equalities of the theorem hold for $e_1$.
Observe that the colour pattern associated with  $\co(\boldsymbol{r}_1)$ can be obtained
from that of   $\co(\boldsymbol{r})$ in  two steps.     
First, we cut the top row, thus, producing a wrong pattern, as the last $r_1-1$ columns begin with a green 
box. Second, these boxes are repainted red. 
The figure bellow illustrates the passage from 
$\co(3,2,1)$ to $\co(4,1)$. 

\vskip1.5ex
\ytableausetup{mathmode, boxsize=1.2em}
\quad
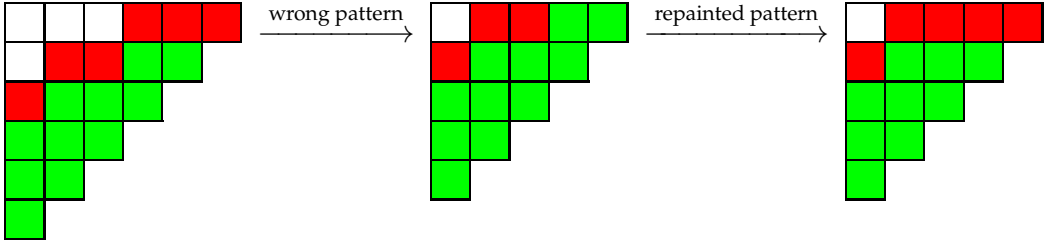
\begin{ytableau}
*(white)& *(white)& *(white)& *(red) & *(red)& *(red)\\
*(white)& *(red) & *(red) & *(green) &*(green) \\
*(red) & *(green) &*(green) &*(green) \\
*(green) &*(green) &*(green) \\
*(green) &*(green) \\
*(green) \\
\end{ytableau} 
\ $\xlongrightarrow{\text{wrong pattern}}$ \
\begin{ytableau}
*(white)& *(red) & *(red)& *(green) & *(green)\\
*(red) & *(green) &*(green) &*(green) \\
*(green) &*(green) &*(green) \\
*(green) &*(green) \\
*(green) \\
\end{ytableau}
\ $\xlongrightarrow{\text{repainted pattern}}$ \
\begin{ytableau}
*(white)& *(red) & *(red)& *(red) & *(red)\\
*(red) & *(green) &*(green) &*(green) \\
*(green) &*(green) &*(green) \\
*(green) &*(green) \\
*(green) \\
\end{ytableau}

 \vskip2ex

Recall from Proposition~\ref{complete} and 
Example~\ref{s(i)-A} that in the pattern related to $\co$ the number of coloured boxes equals $n+\frac{1}{2}\dim \co$ and the number of green boxes equals $\frac{1}{2}\dim\co$.

By the inductive hypothesis, 
\[
 \dim \left<  \textsl{d}_{e_1}  \Delta_k^{[m]} \mid m\ge 1 \right>_{\bbk}=(n-1)+\frac{1}{2}\dim  
  \co(\boldsymbol{r}_1)=\big(n+\frac{1}{2}\dim\co\big)-r_1.
\]
Observe that 
$\textsl{d}_{e}  \Delta_k^{[m]} = \textsl{d}_{e_1}  \Delta_k^{[m]}  \in\gt{gl}_{n-1}$ for $m\ge 1$. 
As can be easily seen, the matrices  $e^k=\textsl{d}_e \Delta_{k+1}$ with $0\le k<r_1$ are linearly 
independent. Furthermore, $e^k=0$ for $k\ge r_1$. 
In order to show that $\textsl{d}_e \eus C$ has the required dimension, 
it is enough  to prove that $$\left< e^0,\ldots,e^{r_1-1}\right>_{\bbk} \cap  \gt{gl}_{n-1}=0.$$ For $0<k<r_1$, we have 
\[
  e^k v_{r_1-k+1}=e v_{r_1}=v_{r_1+1}+v_1 \ \text{ and } 
  e^k v_s \in \left<v_2,\ldots,v_c\right>_{\bbk} \ \text{if } s\ne r_1-k+1.
\] 
Also $e^0 v_{k}=v_k$ for $2\le k\le c$. Since the vectors 
$v_{2},\ldots,v_{r_1}$ are linearly independent,  $$\left< e^0,\ldots,e^{r_1-1}\right>_{\bbk} \cap  \gt{gl}_{n-1}\subset (\bbk e^0 \cap \gt{gl}_{n-1}).$$ 
Clearly,  $e^0\not\in\gt{gl}_{n-1}$. 
Therefore $\dim\textsl{d}_e\eus C=n+\frac{1}{2}\dim\co$.    

The behaviour of $\eus C$ on $\co$ is a more delicate question.  
Recall that $\gt g^e$ is the kernel of the canonical projection 
$T^*_e\gt g^*\to T^*_e\co$. Furthermore, 
we need the following  obvious observations: $\textsl{d}_e \Delta_k\in\gt g^e$ for each $k$ 
and $\dim\co - \dim\co(\boldsymbol{r}_1)=2(r_1-1)$. By the inductive hypothesis, 
the images of  $\textsl{d}_{e}\Delta_k^{[m]}$ with $m>1$ under the projection  
$$\gt{gl}_{n-1}\to \gt{gl}_{n-1}/(\gt{gl}_{n-1})^{e_1}$$ 
span a subspace of dimension $\frac{1}{2} \dim\co(\boldsymbol{r}_1)$. 

Consider now the green elements $\Delta_k^{[1]}$.  
Here $1 \le k \le r_1-1$ and  $\textsl{d}_e \Delta_k^{[1]}=\textsl{d}_{e_1} \Delta_k^{[1]}=e_1^{k-1}$.
Clearly, $e_1^k\in (\gt{gl}_{n-1})^{e_1}$ for each $k$. 
In order to finish the proof it suffices to show that the differentials  
$e_1^k$ with $0\le 0\le r_1-2$ remain  linearly independent on $T_e\co=\ad\!^*(\g)e$.

Let $y\in(\gln)^e$. Using elementary properties of centralisers~\cite[Sect.~1]{surp}, one readily sees that 
$v_{r_1+1}-v_1$ does not lie in 
\[
  R(y)=y\left<v_3,\ldots,v_{r_1},v_{r_1+1}+v_1\right>_{\bbk} +
   \left<v_3,\ldots,v_{r_1},v_{r_1+1}+v_1,v_{r_1+2},\ldots,v_c\right>_{\bbk}.
\] 
Assume that there is a non-trivial linear combination $y=\beta_0 e_1^0 + \ldots + \beta_{r_1-2} e_1^{r_1-2}$ such that $y\in\g^e$. Take the smallest $k\ge 0$ with $\beta_k\ne 0$. Then 
$y v_{r_1+1-k} \in \beta_k v_{r_1+1} + \left<v_{r_1+2},\ldots,v_c\right>_{\bbk}$. 
Here $r_1+1-k\ge 3$ and $v_{r_1+1}-v_1\in R(y)$, a contradiction! 
\end{proof}

\begin{rmk} \label{rmk-limits}
{\sf (i)} The strategy used in the proof of Theorem~\ref{strong} is suggested by a connection between 
MF- and GT-subalgebras. Namely,    
by a result of Vinberg, $\eus C$ can be realised as a limit of MF-subalgebras. That is, if
\[
   a(t)= E_{11}+tE_{22}+\ldots + t^{n-1} E_{nn} , 
\]
then $\lim_{t\to 0} \eus F_{a(t)}=\eus C$ for  the chain as above, see~\cite[6.4]{v:sc}. 
Even more explicitly, in $ \mathbb P \bigl(\gS^{k-m}(\gln)\bigr)$, we have 
$\lim_{t\to 0} \left<\partial^m_{a(t)} \Delta_k\right>=\left<\Delta^{[m]}_{k-m}\right>$, cf.~\cite[Ex.~5.5]{m-y}. 

The properties of $\eus F_{a(t)}$ and its restriction to $\co$, see Proposition~\ref{complete} and 
Example~\ref{s(i)-A}, suggest how to construct bases for  $\textsl{d}_e \eus C$ and 
$\textsl{d}_e (\eus C|_{\co})=(\textsl{d}_e \eus C)/((\gln)^e\cap \textsl{d}_e\eus C)$.   
Indeed, as we have seen in the proof of Theorem~\ref{strong}, the differentials of the coloured 
elements $\Delta_k^{[m]}$ form a basis of $\textsl{d}_e\eus C$.
By the definition of a colour pattern, $\textsl{d}_e (\partial_{a(t)}^{k} H_i) \in(\gln)^e$ for 
the red elements $\partial_{a(t)}^{k} H_i$. From this one can deduce that the differentials 
$\textsl{d}_e \Delta_k^{[m]}$ with red $\Delta_k^{[m]}$ form a basis of 
$\textsl{d}_e\eus C \cap (\gln)^e$. 
The uncoloured elements $\Delta_k^{[m]}$ restrict to zero on $\co$.  \\[.3ex]   \indent
{\sf (ii)} Let $\eus A=\lim_{t\to 0} \eus F_{a(t)}$ with $a(t)\in\gt t$ be a limit in the sense of 
\cite[6.4]{v:sc}. According to \cite{t:max},
$\dim\textsl{d}_x \eus A=\bb(\gt g)$ for each $x\in\eus K$, where 
$\eus K$ is the Kostatn section as in the proof of Theorem~\ref{thm-s}.
Therefore $\eus A$ is complete on any regular orbit, cf. Lemma~\ref{obvious}. 
\end{rmk}

\begin{thm}    \label{gt-com}
The GT-subalgebra $\eus C$ is complete on every adjoint orbit of $G=\GL_n$. 
\end{thm}
\begin{proof}
For a nilpotent orbit $Ge$, the result follows from Theorem~\ref{strong}. 
Proposition~\ref{lm-c-n} immediately extends it to all orbits.  
\end{proof}

\begin{thm}           \label{thm-co}
The action of\/ $\GL_{n-1}$ on each adjoint orbit\/ $\GL_n x\subset \gln$   
is coisotropic.  
\end{thm}
\begin{proof}
By Theorem~\ref{gt-com}, $\eus C$ is complete on every adjoint orbit. More precisely, since 
$\Delta_1,\dots,\Delta_n$ are constant on the orbits, the proper subalgebra 
$\eus C\cap \gS(\gt{gl}_{n-1})$ is complete on every orbit $\GL_n x\subset\gln$. 
This family consists of Noether integrals. 
The discussion in Section~\ref{subsec-cois} and, in particular, assertion {\sf (NF)}  
show that the action of $\GL_{n-1}$ on $\GL_n x$ is coisotropic.
 \end{proof}

\begin{rmk}        \label{rem-el}
{\sf (i)} Note that Theorem~\ref{gt-com} provides a new unusual   
proof of Elashvili's conjecture in type {\sf A}. The 
argument goes as follows. Take  $x\in\gln^*$ such that $(\gln)^x\ne \gln$. Since $\eus C$ is complete on  
$\GL_n x$ and $\eus C=\lim_{t\to 0} \eus F_{a(t)}$, the MF~subalgebra 
$\eus F_{a(t)}$ is complete on $\GL_n x$ for at least one $t\in\bbk^{\times}$. 
Then according to Lemma~\ref{comp-semi}, $\ind(\gln)^x=\rk\gln$.  
\\[.3ex]   \indent
{\sf (ii)} Theorem~\ref{gt-com} has a different, more sophisticated and inductive line of argument 
that does not involve the direct calculation of Theorem~\ref{strong}.
Suppose that the statement holds for $\GL_{n-1}$. Take a nilpotent orbit $Ge\subset \g^*$. 
The Gelfand--Tsetlin subalgebra of $\gS(\gt{gl}_{n-1})$ separates generic $\GL_{n-1}$-orbits on the image $\mu(Ge)\subset\gt{gl}_{n-1}^*$ and is complete 
on each orbit of $\GL_{n-1}$. It can be deduced  
from~\eqref{eq-cois} that the Gelfand--Tsetlin 
subalgebra of $\gS(\gt{gl}_{n-1})$ is complete on $Ge$. Hence $\eus C$ is complete on $Ge$. By Proposition~\ref{lm-c-n}, $\eus C$ is complete on every adjoint orbit.  
\end{rmk}

\subsection{$\bl$-systems}   \label{subs:lambda}
In their approach to GT integrable systems, Guillemin and Sternberg prefer to deal with eigenvalues 
of Hermitian matrices (i.e., piecewise smooth functions)~\cite{gs1}. 
Take the compact form $\gt k=\gt u_n$ and identify $\gt k^*$ with $i\gt u_n$. Now the 
eigenvalues $\{\lambda_k\}$ of $A\in\gt k^*$ are real numbers. 
Let $\bl_k$  with $1\le k\le n$ be the corresponding functions on $\gt k^*$, 
i.e., $\bl_k(A)=\lambda_k$, and likewise for $\bl_k^{[m]}$.  
The completely integrable system on $KA\subset\gt k^*$ is given by the restrictions 
of $\{ \bl_k^{[m]}\mid 1\le m<n \ \& \ 1\le k\le n{-}m\}$. We call it the 
$\bl$-system. 

There is an obvious connection between $\eus C$ and the $\bl$-system. 
Let $\sigma_k$ be the $k$-th elementary symmetric polynomial.   
If one defines $\bl_k$ over $\bbk$ or considers $\Delta_k$ as 
real valued functions on $\gt k^*$, then $\Delta_k=\sigma_k(\bl_1,\ldots,\bl_n)$.  
Take $A\in\gt u_n^*\subset\gt{gl}_n(\mathbb C)^*$. 
Using a standard argument, one proves that 
\begin{equation}      \label{eq-l}
  \text{the $\bl$-system is complete on  ${\rm U}_n A$} \ \Longleftrightarrow  \ \eus C \
\text{ is complete on } {\rm GL}_n(\mathbb C)A. 
\end{equation}
Moreover, we see that there is a connection between  the $\boldsymbol{\lambda}$-system and 
the colour patterns used in the proof of Theorem~\ref{strong}.

Until the end of this section, assume that $\bbk=\mathbb C$ and 
therefore $\GL_n=\GL_n(\mathbb C)$.
Let $\co$ be the dense orbit in the associated cone of $\GL_n A$. 
Then for each $m$,  the number of elements $\bl_k^{[m]}$  with $1\le k \le n-m$ that are functionally 
independent on ${\rm U}_n A$ is equal to the number of 
green elements $\Delta_k^{[m]}$ in the colour pattern associated with $\co$. 
This connection explains also the choice of $e'$ in the proof of Theorem~\ref{strong}. 

Let $\lambda_1\le \ldots\le \lambda_n$ be the eigenvalues of $A$. 
Let $A_1\in\gt u_{n-1}^*$ denote the restriction of $A$ to $\gt u_{n-1}$. 
If $\mu_1\le\ldots\le\mu_{n-1}$ are the eigenvalues of $A_1$, then
$\lambda_i\le \mu_i \le \lambda_{i+1}$.  
For a non-regular orbit ${\rm U}_n A$, $\lambda_{i+1}=\lambda_i$ for some $i$. Therefore, gathering 
together equal eigenvalues of $A$, we get a partition of $n$ different from $(1^n)$. 
The parts of the dual partition,
say $r_1\ge \ldots \ge r_t> 0$, are the sizes of the Jordan blocks of $e\in\co$~\cite{kraft}. 
Suppose that $A$ is a generic representative of  ${\rm U}_n A$. 
The key point in the complete integrability of $\boldsymbol{\lambda}$ on ${\rm U}_n A$~\cite{gs1} is 
that the eigenvalues of $A_1$ are not equal if they do not have to be. In other words, 
the associated cone of $\GL_{n-1} A_1$ is the closure of
$\GL_{n-1} e'$, where $e'$ is given by   
the partition  $(r_1+r_2-1,r_3,\ldots,r_t)$. 

\begin{ex}
Let $A\in\gt u_7^*$ have the eigenvalues 
\[
  \lb_1=\lb_2=\lb_3<\lb_4=\lb_5<\lb_6=\lb_7\,.
\]
This means that  $\mu_1=\mu_2$, but there are no other necessary equalities among the eigenvalues of $A_1$. In terms of partitions, this set of eigenvalues gives rise to
the partition $(3,2,2)$, with the dual partition $\boldsymbol{r}=(3,3,1)$. 
Then $\boldsymbol{r}_1=(5,1)$, and its dual is $(2,1,1,1,1)$. This last partition describes the coincidence
of 
the eigenvalues of $A_1$.

On the orbit ${\rm U}_7 A$, we have $\mu_1=\mu_2=\lambda_1$ as well as $\mu_4=\lambda_4$ and 
$\mu_6=\lambda_6$. 
Among the function $\bl_k^{[1]}$, only two, namely $\bl_3^{[1]}$ and $\bl_5^{[1]}$, are
functionally independent. 
According to the colour pattern used in the proof of Theorem~\ref{strong}, the images of the differentials 
$\textsl{d}_A\boldsymbol{\lambda}_k^{[1]}$ with $1\le k\le 6$ 
span a subspace of dimension $2$  in the quotient of $T^*_A \gt u_7^*$ by $\gt u_7^A$. 
\end{ex}

\section{Corank on closures of sheets and the orthogonal case}
\label{sect:5}

Let $(M,\omega)$ and $Q$ be as in Section~\ref{subsec-cois}.  
Set $U:=\{y\in M \mid \dim(\q y)=\max_{x\in M}\dim(\q x)\}$.

\begin{df}\label{def-def}
The {\it defect} of the $Q$-action on $M$ is 
\[
    {\sf def}(M)={\sf def}_Q(M)=\min_{y\in U} \dim (\q y \cap (\q y)^\perp) ;
\]
and the {\it corank} of the $Q$-action is
$
   \cork(M)=\cork_Q(M):=\max_{y\in U}\rk({\omega_y}\vert_{(\q y)^{\perp}})
$.
\end{df}
If $x\in U$, then $\cork(M)=\dim M-\dim(\q x)-{\sf def}(M)$.
We omit the indication of $Q$ if it is clear from the context.
The coisotropic actions are of corank zero.  

From now on, suppose that $M$ is an irreducible algebraic variety defined over $\bbk$. Then 
the image $\mu(M)\subset \q^*$ is a $Q$-stable subset, which is dense in its closure.  
Moreover, $\overline{\mu(M)}$ is irreducible. 
For an irreducible $Q$-stable closed subset $Y\subset \q^*$, set 
\begin{equation}          \label{bY}
\bb(Y)= \dim Y - \frac{1}{2} \max_{y\in Y} \dim (\q y). 
\end{equation}
The transcendence degree of a Poisson-commutative subalgebra of $\bbk(Y)$ is bounded above 
by $\bb(Y)$. 
Note that $\bb(\q^*)=\bb(\q)$ is just the ``magic number''. Note also that 
$\max_{y\in\overline{\mu(M)} }\dim(\gt q y)=\max_{y\in\mu(M)} \dim (\q y)$. 
Set $\bb(\mu(M))=\bb(\overline{\mu(M)})$. 

The equality  $\ker\textsl{d}_x\mu =(\gt q x)^\perp$ that has been discussed in 
Section~\ref{subsec-cois}  leads to the following formulas:
\begin{align}
 &  \dim \overline{\mu(M)}=\dim (\q x) \ \text{ for } \ x\in U;   \label{cork-dim}  \\
 &  \max_{y\in \mu(M)} \dim (\q y)  = \dim (\q x)-{\sf def}(M) \ \text{ for } \ x\in U; \label{cork-max}\\
 &  2\bb(\mu(M)) + \cork(M) =  2\dim (\q x) -  \dim (\q x) + {\sf def}(M)  + \cork(M) = \dim M.  \label{eq-cork-pr} 
\end{align} 

By \cite{int}, for any $Q$-stable irreducible closed subset $Y\subset\q^*$, there is 
a subalgebra $\eus A\subset \gS(\q)$ such that  $\{\eus A,\eus A\}$ vanishes on $Y$ and
$\trdeg(\eus A|_Y)=\bb(Y)$.   If   $Y=\overline{\mu(M)}$ and the action of $Q$ on $M$ is coisotropic, then  
the pull-back $\mu^*(\eus A)$ contains a complete  family of functions, Noether integrals,  on  $M$.   
We only need these statements if $Q$ is reductive, in which case the proof simplifies drastically. 

\begin{lm}[{cf. \cite[Sect.~3]{int}}]      \label{int-N}
Suppose that $Q$ is reductive. Then there is $a\in\q^*$ such that 
$\trdeg(\eus F_a|_Y)=\bb(Y)$ for the MF-subalgebra $\eus F_a$ associated with $a$. 
\end{lm} 
\begin{proof}
Since $Y\subset\q^*$, each fibre of the quotient map $Y\to Y/\!\!/ Q$ contains an open orbit. 
Therefore $\dim Y/\!\!/ Q = \dim Y - r$, where $r=\max_{y\in Y}\dim (\gt q y)$. 
Hence also $\dim\textsl{d}_y(\gS(\q)^{\q})=\dim Y-r$ for generic $y\in Y$. 
Fix one $y\in Y$ having this property. 
There is $a\in\q^*$ such that $\eus F_a$ is complete on $Qy$, see~\cite{bols} and 
Section~\ref{sect:2}.  Since $\gS(\q)^{\q}\subset\eus F_a$, we conclude that 
$\trdeg({\eus F_a}|_Y)=r+\frac{1}{2}\dim(Qy)=\bb(Y)$. 
\end{proof}

\subsection{Numerical invariants of sheets}
Let $H$ be an arbitrary reductive subgroup of a connected reductive group  $G$. Let $S\subset\g$ 
be a $G$-sheet and  $Ge$ the unique nilpotent orbit in $S$, see~\cite[Sect.~5.8,\,Kor.(a)]{bokr}.  
For any coadjoint orbit $Gx\subset\g^*$, the moment map w.r.t. $H$,
$\mu\!: Gx\to \gt h^*$, is given by the restriction $\g^*\to\gt h^*$ of linear functions. 
The dual map (co-morphism) $\mu^*$ is the canonical inclusion $\gS(\gt h)\subset \gS(\g)$.

\begin{lm}      \label{cork-sheet}
For any $G$-orbit ${\mathcal O}\subset S$, one has 
$\cork_H({\mathcal O})\le \cork_H(Ge)$.
\end{lm}
\begin{proof}
Set $Y=\overline{\mu(Ge)}$. This is an $H$-stable irreducible closed subset of $\gt h^*$.
By Lemma~\ref{int-N}, there is $a\in\gt h^*$ such that 
$\trdeg({\eus F_a}|_Y)=\bb(Y)$ for the MF-subalgebra $\eus F_a\subset\gS(\gt h)$.
Note that $\eus F_a$ is a homogeneous  Poisson-commutative subalgebra of $\gS(\g)$. 
For each 
$Gx\subset S$, the orbit $Ge$ is  dense in the associated cone  of $Gx$. 
Making use of \eqref{eq-con2}, we write 
\[
    \bb(Y) = \trdeg({\eus F_a}|_{Ge}) \le \trdeg({\eus F_a}|_{Gx}) \le \bb(\mu(Gx)).
\]
By \eqref{eq-cork-pr}, we have 
$\cork_H(M)=\dim M-2\bb(\mu(M))$. Since $\dim(Gx)=\dim(Ge)$, the result follows. 
\end{proof}

\begin{lm}       \label{cork-r}
For any $G$-orbit ${\mathcal O}\subset \g^*$, the corank 
$\cork_H({\mathcal O})$ is equal to the rank of $\hat x$ on 
$\textsl{d}_x (\gS(\g)^H)$ for a generic $x\in{\mathcal O}$.  
\end{lm}
\begin{proof}
By the definition, $\cork_H(M)=\max_{y\in U} \rk({\omega_y}|_{(\gt hy)^\perp})$.  
This number is the {\it rank of the Poisson bracket on 
$\bbk(M)^H$}. Suppose that $F_1,\ldots,F_k\in\bbk({\mathcal O})^H$ are algebraically independent 
and $k=\trdeg \bbk({\mathcal O})^H$. Whenever all $\textsl{d}_y F_i$ are defined for 
$y\in{\mathcal O}$, set $$V(y)=\left<\textsl{d}_y F_i \mid 1\le i\le k\right>_{\bbk}.$$ 
Then $\cork_H({\mathcal O})=\max_{y\in{\mathcal O}} \rk (\hat y |_{V(y)})$.
In \cite[Prop.~2.9]{av}, it is explained how to deduce from results of \cite{Losev} the fact that
$\bbk({\mathcal O})^H={\rm Quot}(\bbk[{\mathcal O}]^H)$.  
By \cite[Lemma~3.7]{bokr}, $\bbk[{\mathcal O}]$ is an integral extension of 
$\bbk[\overline{\mathcal O}]$.  Hence  
$\bbk[{\mathcal O}]^H$ is an algebraic  extension of 
$\bbk[\overline{\mathcal O}]^H$.
Summing up, $\trdeg\bbk({\mathcal O})^H=\trdeg\bbk[\overline{\mathcal O}]^H$.

Since $H$ is reductive, $\bbk[\overline{\mathcal O}]^H$ is the image of $\bbk[\g^*]^H$ under the 
restriction to $\overline{\mathcal O}$. 
Hence $V(y)=\textsl{d}_y (\gS(\g)^H) / \g^y$ on a non-empty open subset of ${\mathcal O}$.
Since $\g^y$ is the kernel of $\hat y$, the result follows. 
\end{proof}

\begin{thm}      \label{thm-d}  
Let $S\subset \g^*$ be a sheet.
\begin{itemize}
\item[\sf (i)] \ The corank of the $H$-action on $G$-orbits
does not change along $S$; 
\item[\sf (ii)] \ if a $G$-orbit $\mathcal O$ lies in $\overline{S}$, then 
$\cork_H({\mathcal O})\le\cork_H(Gx)$ with $x\in S$.  
\end{itemize}
\end{thm}
\begin{proof}
Lemma~\ref{cork-r} readily implies that there is a dense subset of $S$ such that 
$\cork(Gx)=r$ for each orbit $Gx$ in this subset and 
$\cork({\mathcal O})\le r$ for each orbit $\mathcal O\subset \overline{S}$. 

Making use of Lemma~\ref{cork-sheet}, we show that $r\le \cork(Ge)\le r$. Hence $\cork(Ge)=r$. 
Finally  suppose that $Gy\subset S$ is not nilpotent. 
Then $Ge\subset \overline{\bbk^{\times} Gy}$ and in view of  Lemma~\ref{cork-r}
$\cork(Gy)\ge \cork(Ge)=r$. At the same time $\cork(Gy)\le r$.  
This finishes the proof. 
\end{proof}

There are many other characteristics of $H$-actions that 
do not change along a sheet. 

\begin{thm}
Let $S\subset\g$ be a sheet with unique nilpotent orbit $Ge$. Take $Gx\subset S$. Then 
\begin{itemize}
\item[(1)] \ $\trdeg\bbk[Gx]^H=\trdeg\bbk[Ge]^H$; 
\item[(2)] \ $\max_{x'\in Gx} \dim(Hx') = \max_{e'\in Ge} \dim (He')$;
\item[(3)] \ $\dim\mu(Gx)=\dim\mu(Ge)$;
\item[(4)] \ ${\sf def}(Gx)={\sf def}(Ge)$; 
\item[(5)] \ $\max_{\xi\in\mu(Gx)} \dim(H\xi) = \max_{\eta\in\mu(Ge)} \dim (H\eta)$.
\end{itemize}
\end{thm}
\begin{proof}
We can safely assume that $x\not\in Ge$ and therefore is not nilpotent. 
Let $F\in\bbk[\g^*]$ be a homogenous $G$-invariant  that is non-zero on $Gx$. Then 
$\overline{Ge}/\!\!/H$ is defined as the zero set of $F$ in $\overline{\bbk^{\times}Gx}/\!\!/H$. 
Hence $\dim \overline{Gx}\md H=\dim\overline{Ge}/\!\!/H$. As we have seen in the proof of Lemma~\ref{cork-r}, $\trdeg\bbk[Gy]^H =\trdeg\bbk[\overline{Gy}]^H$ for each orbit. 
This settles (1). 

By \cite[Prop.~2.9]{av}, we have $\bbk(Gy)^H={\rm Quot}(\bbk[Gy]^H)$ for each orbit. 
Hence the dimension of a generic $H$-orbit on $Gy$  is equal to 
$\dim(Gy)-\trdeg\bbk[Gy]^H$. Thus, (1) implies~(2).  

The dimension of  $\mu(Gy)$ is equal to the dimension of 
a generic $H$-orbit on $Gy$, see \eqref{cork-dim}. Therefore it does not change along a sheet either. 

The defect of a Hamiltonian action can be expressed via the corank
\[
         {\sf def}(Gy)= \dim(Gy) - \max_{y'\in Gy} \dim (H y') - \cork(Gy').
\] 
In view of (3) and  Theorem~\ref{thm-d},  the  defect does not change, ${\sf def}(Gx)={\sf def}(Ge)$. 

Finally,
$\max_{\xi\in\mu(Gy)} \dim (H\xi) = \dim\mu(Gy) - {\sf def}(Gy)$, see \eqref{cork-max}.
\end{proof}
 
Of course, there are examples such that $\mu(Gx)\ne \mu(Ge)$. 
 
\begin{ex}
Consider  $(\g,\gt h)=(\gt{sl}_3,\gt{sl}_2)$ and take $x={\rm diag}(1,1,-2)$. Then
$e$ is a minimal nilpotent element.  Here $\mu(Gx)$ is the $\SL_2$-orbit 
of ${\rm diag}(1,-1)$, and $\mu(Ge)$ is the null-cone in $\gt{sl}_2$. 
We have $\dim(Gx)=4$ and  $\dim\mu(Gx)=\dim\mu(Ge)=2$.  Further, $\bb(\mu(Gx))=1=\bb(\mu(Ge))$. 
The $\SL_2$-action on $Gx$ and on  $Ge$ has corank $1$. 
\end{ex} 
  
\begin{ex} 
 Take $G=\GL_3$, $H=\GL_2$, $x={\rm diag}(2,2,1)$. 
 Here the $H$-action on each $Gy\subset \g^*$ is coisotropic.    
We have
\[
\mu(Gx)=\bigcup_{a\in\bbk} H\left(\begin{array}{cc} 2 & 0 \\ 0 & a \end{array}\right) \cup H\left(\begin{array}{cc} 2 & 1 \\ 0 & 2 \end{array}\right).
\]  
Further,  $e$ is conjugate to $E_{12}$ in $\mathfrak{sl}_3$ and 
\[
\mu(Ge)=\bigcup_{a\in\bbk} H\left(\begin{array}{cc} 0 & 0 \\ 0 & a \end{array}\right) \cup H\left(\begin{array}{cc} 0 & 1 \\ 0 & 0 \end{array}\right).
\] 
\end{ex}

\subsection{The orthogonal case}
There are  the orthogonal versions of the Gelfand--Tsetlin subalgebra and the $\bl$-system of 
Guillemin--Sternberg. Suppose that $\g=\son=\gt{so}_n(\bbk)$. Fix a sequence 
\[
   \gt{so}_n \supset \gt{so}_{n-1}\supset \ldots \supset \gt{so}_3\supset \gt{so}_2.
\]
Let $\eus C\subset\gS(\g)$ be the subalgebra generated by $\gS(\gt{so}_m)^{\gt{so}_m}$ with 
$n\ge m\ge 2$. Then $\eus C$ is the image in $\gS(\g)$ of the  famous commutative GT-subalgebra of 
 $\eus U(\g)$~\cite{gt-2}.  Hence  $\{\eus C,\eus C\}=0$. 
Similar to the $\gln$ case, $\eus C$ has $\bb(\g)$ algebraically independent generators.  
Comparing Poincar\'e series 
one can prove that in the orthogonal case, the GT-subalgebra $\eus C$ cannot be realised 
as a limit of MF-subalgebras. Nevertheless, our results in~\cite[Sect.~6.2]{oy} show that 
$\eus C$ is a maximal Poisson-commutative subalgebra of $\gS(\g)$.

With the obvious changes, one defines strongly regular and strongly nilpotent elements, as well as 
the $\bl$-system related to eigenvalues. In the orthogonal case, there are no strongly nilpotent elements 
$e$ such that $\dim\textsl{d}_e \eus C=\bb(\g)$ if $n\ge 4$, see \cite[Prop.~5.14]{EvC}. Theorem~4.17 
of  that paper asserts  that $\eus C$ is complete on each {\sl regular} coadjoint orbit. 
We prove that $\eus C$ is complete on {\sl each} coadjoint orbit, lifting the assumption that the orbit is regular.  

\begin{thm}               \label{orth}
For any $x\in\son$, the GT-subalgebra $\eus C\subset \gS(\son)$ is complete on every (co)adjoint orbit 
$\SO_n x$ and the action of\/ $\SO_{n-1}$  on\/  $\SO_n x$ is coisotropic.   
\end{thm}
\begin{proof}
Assume that both statements are true for $\SO_{n-1}$. The base of induction is the case $n=2$, where the assertions  are obvious.  

Since $(G,H)=(\SO_n,\SO_{n-1})$ is a strong Gelfand pair, the action of $\SO_{n-1}$ on $G/B$ is spherical, see Remark~\ref{SGP-cond}.
Hence the action of $\SO_{n-1}$ on $T^*(G/B)$ and its image under the moment map
$\mu\!: T^*(G/B) \to \g^*$ is coisotropic, see Section~ \ref{sec-flag} and \cite[Sect.~2.3]{av}. The regular nilpotent orbit $G\boldsymbol{e}\subset \g^*$ is dense in this image. 
Therefore the action of $\SO_{n-1}$ on $G\boldsymbol{e}$ is coisotropic, cf. \cite[Thm~2.6]{av}.   The same can be said about any Richardson orbit. 
However, not every nilpotent orbit in $\gt{so}_n$ is Richardson. 

The unique sheet containing $G\boldsymbol{e}$ is $\g^*_{\sf reg}$. 
By Lemma~\ref{cork-sheet}, $\cork_H({\mathcal O})=0$ for each 
${\mathcal O}\subset \g^*_{\sf reg}$. Theorem~\ref{thm-d}  extends this fact to all orbits,
cf. \cite[Prop.~2.7]{av}. 

Next we need to go through  the standard inductive argument used, for example, in~\cite{gs2}. 
Let $\eus C^{[1]}$ be the Gelfand--Tsetlin subalgebra in $\gS(\gt h)$ and 
$\eus C^{[2]}$  --- in $\gS(\gt{so}_{n-2})$. 
Set $Y=\overline{\mu(Gx)}$. 
Each fibre of the quotient map $Y\to Y/\!\!/ H$ contains an open orbit. 
Therefore $\dim Y/\!\!/ H = \dim Y - r$, where $r=\max_{y\in Y}\dim (\gt h y)$. 
Hence also $\dim\textsl{d}_y(\gS(\gt h)^{\gt h})=\dim Y-r$ for generic $y\in Y$. 
By induction, $\eus C^{[1]}$ is complete on each $Hy\subset Y$. 
More precisely, $\eus C^{[2]}$ is complete on $Hy$. 
Since $\gS(\gt h)^{\gt h}\subset \eus C^{[1]}$, we have 
\[
        \dim\textsl{d}_y\eus C^{[1]}=(\dim Y-r)+\frac{r}{2} =\bb(Y)
\] 
for generic $y\in Y$. 
Since the action of $H$ on $Gx$ is coisotropic, we have $\bb(Y)=\frac{1}{2}\dim (Gx)$ 
by~\eqref{eq-cork-pr} and thereby $\eus C^{[1]}$ is complete on $Gx$. 
Thus, $\eus C$ is complete on $Gx$.  
\end{proof}

\end{document}